\documentclass[onefignum,onetabnum]{parts/siamart220329}

\usepackage{amsfonts, amssymb, amsopn}
\usepackage{bm}
\usepackage{graphicx}
\usepackage{epstopdf}
\usepackage{algorithmic}
\ifpdf
  \DeclareGraphicsExtensions{.eps,.pdf,.png,.jpg}
\else
  \DeclareGraphicsExtensions{.eps}
\fi

\renewenvironment{equation*}{\[}{\]\ignorespacesafterend}

\newsiamremark{remark}{Remark}
\newsiamremark{example}{Example}
\newsiamremark{hypothesis}{Hypothesis}
\crefname{hypothesis}{Hypothesis}{Hypotheses}
\newsiamthm{claim}{Claim}

\headers{Entrywise tensor-train approximation}{Stanislav Budzinskiy}

\title{Entrywise tensor-train approximation of large tensors via random embeddings\thanks{This research was funded in whole by the Austrian Science Fund (FWF) \href{https://doi.org/10.55776/F65}{10.55776/F65}. For open access purposes, the author has applied a CC BY public copyright license to any author-accepted manuscript version arising from this submission.}}

\author{
Stanislav Budzinskiy\thanks{Faculty of Mathematics, University of Vienna, Kolingasse 14-16, 1090 Vienna, Austria (\email{stanislav.budzinskiy\allowbreak@univie.ac.at}).}
}

\newcommand{\Real}{\mathbb{R}}
\newcommand{\N}{\mathbb{N}}

\newcommand{\Frob}{\mathrm{F}}
\newcommand{\TT}{\mathrm{TT}}
\newcommand{\Ell}{\mathrm{L}}

\newcommand{\Matrix}[1]{\bm{\mathrm{#1}}}
\newcommand{\Tensor}[1]{\bm{\mathcal{#1}}}
\newcommand{\trans}{\intercal}

\newcommand{\Pind}[1]{\bm{#1}}
\newcommand{\PindSet}{\bm{\mathfrak{I}}}
\DeclareMathOperator{\PindJoin}{\dot{\times}}
\DeclareMathOperator{\PindCat}{\dot{+}}
\newcommand{\PindKron}[2]{\overline{#1 #2}}
\newcommand{\Compl}[1]{{#1}^{\mathsf{c}}}
\newcommand{\Part}{\bm{\mathfrak{P}}}

\newcommand{\set}[3][]{#1\{ #2 : #3 #1\}}
\newcommand{\Proj}[3][]{\Pi_{#2} #1( #3 #1)}

\newcommand{\Norm}[3][]{#1\| #2 #1\|_{#3}}

\newcommand{\Dotp}[4][]{#1\langle #2, #3 #1\rangle_{\mathrm{#4}}}

\newcommand{\rank}[2][]{\mathrm{rank} #1( #2 #1)}
\newcommand{\ttrank}[2][]{\mathrm{rank_{TT}} #1( #2 #1)}

\newcommand{\Expectation}{\mathbb{E}}
\newcommand{\Probability}{\Pr}

\newcommand{\Prob}[2][]{\Probability #1\{ #2 #1\}}

\ifpdf
\hypersetup{
  pdftitle={Entrywise tensor-train approximation of large tensors via random embeddings},
  pdfauthor={Stanislav Budzinskiy}
}
\fi

\begin{document}

\maketitle

\begin{abstract}
The theory of low-rank tensor-train approximation is well understood when the approximation error is measured in the Frobenius norm. The entrywise maximum norm is equally important but is significantly weaker for large tensors, making the estimates obtained via the Frobenius norm and norm equivalence pessimistic or even meaningless. In this article, we derive a direct estimate of the entrywise approximation error that is applicable in some of these cases. The estimate is given in terms of the higher-order generalization of the matrix factorization norm, and its proof is based on the tensor-structured Hanson--Wright inequality. The theoretical results are accompanied by numerical experiments carried out with the method of alternating projections.
\end{abstract}

\begin{keywords}
tensor train, maximum norm, factorization norm, Hanson--Wright inequality, coherence, alternating projections
\end{keywords}

\begin{MSCcodes}
15A23, 15A69, 60F10, 65F55
\end{MSCcodes}

\section{Introduction}

The entrywise storage and processing of $n_1 \times \cdots \times n_d$ tensors become expensive when $n = \max\{n_1, \ldots, n_d\}$ is large. This fundamental difficulty, common in data-intensive areas such as scientific computing and data analysis, necessitates the use of low-parametric tensor representations. The reduction of complexity, however, comes at a typically inevitable cost of introducing errors in the entries of the tensor. It is therefore important to have a priori estimates that relate the compression rate to the achievable approximation errors for a specific representation format.

Tensor networks are among the most widespread and effective low-parametric representations: they leverage the idea of separation of variables and generalize low-rank matrix factorizations. A particularly successful example of a tensor network is the tensor-train (TT) format \cite{oseledets2009breaking, oseledets2011tensor}, also known as matrix product states \cite{vidal2003efficient, schollwock2011density}, which reduces the storage requirements from $\mathcal{O}(n^d)$ to $\mathcal{O}(d n r^2)$. The effective rank, $r$, measures the complexity of the TT representation.

The known a priori estimates of the error induced by TT approximation are formulated for the Frobenius norm $\Norm{\Tensor{A}}{\Frob}^2 = \sum_{i_1, \ldots, i_d} |\Tensor{A}(i_1, \ldots, i_d)|^2$ in terms of the singular values of the unfolding matrices \cite{oseledets2011tensor}. An estimate in the maximum norm $\Norm{\Tensor{A}}{\max} = \max_{i_1, \ldots, i_d} |\Tensor{A}(i_1, \ldots, i_d)|$ can be deduced from the norm inequality $\Norm{\Tensor{A}}{\max} \leq \Norm{\Tensor{A}}{\Frob}$, but becomes too pessimistic if the singular values decay slowly. We aim to derive a new a priori estimate of the entrywise TT approximation error given in terms of the \textit{higher-order generalization of the matrix factorization norm}.

\subsection{Matrix factorization norm and random embeddings}
For a matrix $\Matrix{A} \in \Real^{n_1 \times n_2}$ and $1 \leq p,q \leq \infty$, let $\Norm{\Matrix{A}}{p,q} = \sup\set{\Norm{\Matrix{Ax}}{q}}{\Norm{\Matrix{x}}{p} = 1}$, $\Norm{\Matrix{A}}{p} = \Norm{\Matrix{A}}{p,p}$, be the operator norm induced by the $\ell_p$ and $\ell_q$ vector norms. Less known is
\begin{equation}
\label{eq:gamma2}
    \gamma_2(\Matrix{A}) = \inf\set{\Norm{\Matrix{X}}{2,\infty} \Norm{\Matrix{Y}}{2,\infty}}{\Matrix{X}, \Matrix{Y}~\text{such that}~\Matrix{A} = \Matrix{X} \Matrix{Y}^\trans}.
\end{equation}
One can show that $\gamma_2$ is a norm, which we shall call the \textit{factorization norm}\footnote{It is also known as the max-norm \cite{srebro2005rank}, which should not be confused with $\Norm{\cdot}{\max}$ that we use.} \cite{linial2007lower}. It originates in~the theory of Banach spaces \cite[\S 13]{tomczak1989banach}, where a linear operator between Banach spaces $\mathfrak{B}_1 \to \mathfrak{B}_2$ is factorized into two linear operators through a Lebesgue space $\mathfrak{B}_1 \to \Ell_p \to \mathfrak{B}_2$. In \cref{eq:gamma2}, the matrix $\Matrix{A}$ is factorized as $(\Real^{n_2}, \ell_1) \to (\Real^k, \ell_2) \to (\Real^{n_1}, \ell_\infty)$ with an arbitrary $k \in \N$ (note that $\Norm{\Matrix{Y}}{2,\infty} = \Norm{\Matrix{Y}^\trans}{1,2}$).

\begin{remark}
The factorization norm is used in matrix completion \cite{srebro2004maximum, lee2010practical, cai2016matrix, foucart2020weighted} and communication complexity \cite{linial2007complexity, linial2007lower, matouvsek2020factorization}. By Grothendieck's inequality \cite{pisier2012grothendieck, friedland2019elementary}, its dual $\gamma_2^\ast(\Matrix{A}) = \sup\set{\Dotp{\Matrix{A}}{\Matrix{B}}{\Frob}}{\gamma_2(\Matrix{B}) = 1}$ satisfies $\Norm{\Matrix{A}}{\infty,1} \leq \gamma_2^\ast(\Matrix{A}) \leq K_G \Norm{\Matrix{A}}{\infty,1}$ with an absolute constant $K_G > 0$. It follows that $\gamma_2(\Matrix{A}) \leq \Norm{\Matrix{A}}{\infty,1}^\ast \leq K_G \gamma_2(\Matrix{A})$, where the dual norm of $\Norm{\cdot}{\infty,1}$ is equal to \cite{linial2007lower, johnston2012norms}
\begin{equation*}
    \Norm{\Matrix{A}}{\infty,1}^\ast = \inf \set[\Big]{\sum\nolimits_{\alpha} |\sigma_\alpha|}{\Matrix{A} = \sum\nolimits_\alpha \sigma_\alpha \Matrix{u}_\alpha \Matrix{v}_\alpha^\trans, ~\Norm{\Matrix{u}_\alpha}{\infty} = \Norm{\Matrix{v}_\alpha}{\infty} = 1}.
\end{equation*}
This is an analogue of the nuclear norm $\Norm{\Matrix{A}}{2}^\ast$ with $\ell_\infty$-normalized $\{ \Matrix{u}_\alpha \}$ and $\{ \Matrix{v}_\alpha \}$.
\end{remark}

The quality of low-rank approximation of a matrix in the maximum norm was related to its factorization norm in \cite{srebro2005rank}.

\begin{theorem}[{\cite[Lemma~5]{srebro2005rank}}]
\label{theorem:matrix_approx}
Let $\varepsilon \in (0,1)$ and $n_1, n_2 \in \N$. Consider
\begin{equation}
\label{eq:matrix_rank}
    r = \left\lceil  9 \log\left(3 n_1 n_2\right) / \varepsilon^2 \right\rceil \in \N.
\end{equation}
For every $\Matrix{A} \in \Real^{n_1 \times n_2}$, there exists $\Matrix{B} \in \Real^{n_1 \times n_2}$ of $\rank{\Matrix{B}} \leq r$ such that
\begin{equation*}
    \Norm{\Matrix{A} - \Matrix{B}}{\max} \leq \varepsilon \cdot \gamma_2(\Matrix{A}).
\end{equation*}
\end{theorem}

The proof of \cref{theorem:matrix_approx} is based on taking a factorization $\Matrix{A} = \Matrix{X} \Matrix{Y}^\trans$, drawing a random matrix $\Matrix{R}$, and applying the Johnson--Lindenstrauss lemma to show that $\Matrix{A} \approx (\Matrix{X}\Matrix{R}) (\Matrix{Y}\Matrix{R})^\trans$ entrywise with constant probability. \cref{theorem:matrix_approx} was rediscovered with specific bounds on $\gamma_2(\Matrix{A})$ such as $\gamma_2(\Matrix{A}) \leq \Norm{\Matrix{A}}{2}$ in \cite{udell2019big} and $\gamma_2(\Matrix{A}) \leq \Norm{\Matrix{A}}{\max}$ for symmetric positive-semidefinite matrices in \cite{alon2013approximate}. A different proof based on the Hanson--Wright inequality was proposed in \cite{budzinskiy2024distance}.

\begin{example}
\label{ex:identity_matrix}
Let $\Matrix{I}_n \in \Real^{n \times n}$ be the identity matrix. For each $r \in [n-1]$, its best rank-$r$ approximation in the Frobenius norm satisfies $\Norm{\Matrix{I}_n - \Matrix{I}_n^{(r)}}{\max} = 1$, which is equivalent to rank-zero entrywise approximation. Meanwhile, $\gamma_2(\Matrix{I}_n) \leq 1$ and \cref{theorem:matrix_approx} guarantees that for any $\varepsilon \in (0,1)$ there exists a matrix $\Matrix{B}_n \in \Real^{n \times n}$ of rank bounded by \cref{eq:matrix_rank} such that $\Norm{\Matrix{I}_n - \Matrix{B}_n}{\max} \leq \varepsilon$. The size $n$ needs to be sufficiently large to ensure that $\Matrix{B}_n$ is not full-rank (see \Cref{sec:experimental_settings,tab:rank_values}).
\end{example}

\subsection{Tensor-train factorization quasinorm}

We aim to extend \cref{theorem:matrix_approx} to order-$d$ tensors ($d \geq 2$). A tensor $\Tensor{A} \in \Real^{n_1 \times \cdots \times n_d}$ is said to be represented\footnote{Every tensor admits a TT factorization; it can be obtained with the TT-SVD algorithm \cite{oseledets2011tensor}.} in the TT format with TT cores $\{ \Tensor{G}_s \}_{s=1}^{d}$ of size $\Tensor{G}_s \in \Real^{r_{s-1} \times n_s \times r_s}$ if
\begin{equation}
\label{eq:tt}
    \Tensor{A}(i_1, \ldots, i_d) = \sum_{\alpha_1 = 1}^{r_1} \cdots \sum_{\alpha_{d-1} = 1}^{r_{d-1}} \Tensor{G}_1(1, i_1, \alpha_1) \Tensor{G}_2(\alpha_1, i_2, \alpha_2) \ldots \Tensor{G}_d(\alpha_{d-1}, i_d, 1).
\end{equation}
In short, $\Tensor{A} = \TT(\Tensor{G}_1, \ldots, \Tensor{G}_d)$. Here, $r_0 = r_d = 1$, and the tuple $(r_1, \ldots, r_{d-1}) \in \N^{d-1}$ is called the TT rank of the factorization. The smallest (componentwise) TT rank among all exact factorizations of $\Tensor{A}$ is called its TT rank and denoted by $\ttrank{\Tensor{A}}$.

To generalize $\gamma_2(\Matrix{A})$, recall that $\Norm{\Matrix{A}}{2,\infty} = \max_{i \in [n_1]} \Norm{\Matrix{A}(i,:)}{2}$ for $\Matrix{A} \in \Real^{n_1 \times n_2}$. For a third-order tensor $\Tensor{G} \in \Real^{n_1 \times n_2 \times n_3}$ and a matrix norm $\Norm{\cdot}{\circ}$, we define
\begin{equation*}
    \Norm{\Tensor{G}}{\circ,\infty} = \max_{i \in [n_2]} \Norm{\Tensor{G}(:, i, :)}{\circ}.
\end{equation*}
This is a norm, and we use it to introduce an analogue of \eqref{eq:gamma2} for order-$d$ tensors:
\begin{equation}
\label{eq:gammatt}
    \gamma_{\circ}^{\TT}(\Tensor{A}) = \inf\set[\Big]{\prod\nolimits_{s = 1}^{d} \Norm{\Tensor{G}_s}{\circ,\infty}}{\{ \Tensor{G}_s \}_{s=1}^{d}~\text{such that}~\Tensor{A} = \TT(\Tensor{G}_1, \ldots, \Tensor{G}_d)}.
\end{equation}
We will work with $\gamma_{\Frob}^{\TT}$ and call it the \textit{TT factorization (Frobenius) quasinorm}.

\begin{remark}
\label{remark:cp}
A factorization quasinorm based on the canonical polyadic (CP) tensor format \cite{hitchcock1927expression, kolda2009tensor} was introduced in the context of tensor completion \cite{ghadermarzy2019near, harris2021deterministic}. The TT factorization quasinorm $\gamma_{\Frob}^{\TT}$ is more general as the CP format is a specific type of the TT format \cite{oseledets2011tensor}. See \cite{cao20241} for a different higher-order extension of \cref{eq:gamma2}.
\end{remark}

\begin{lemma}
\label{lemma:tt_factorization_quasinorm}
If $\Norm{\Matrix{A}\Matrix{B}}{\circ} \leq \Norm{\Matrix{A}}{\circ} \Norm{\Matrix{B}}{\circ}$ for all compatible $\Matrix{A}$, $\Matrix{B}$, then $\gamma_{\circ}^{\TT}$ satisfies
\begin{enumerate}
    \item $\gamma_{\circ}^{\TT}(\Tensor{A}) = 0$ if and only if $\Tensor{A} = 0$,
    \item $\gamma_{\circ}^{\TT}(c\Tensor{A}) = |c| \cdot \gamma_{\circ}^{\TT}(\Tensor{A})$ for every $c \in \Real$ and $\Tensor{A} \in \Real^{n_1 \times \cdots \times n_d}$,
    \item $\gamma_{\Frob}^{\TT}(\Tensor{A} + \Tensor{B}) \leq 2^{d/2-1} \cdot (\gamma_{\Frob}^{\TT}(\Tensor{A}) + \gamma_{\Frob}^{\TT}(\Tensor{B}))$ for all $\Tensor{A}, \Tensor{B} \in \Real^{n_1 \times \cdots \times n_d}$,
    \item $\gamma_{2}^{\TT}(\Tensor{A} + \Tensor{B}) \leq \max \{ \gamma_{2}^{\TT}(\Tensor{A}), \gamma_{2}^{\TT}(\Tensor{B}) \}$ for all $\Tensor{A}, \Tensor{B} \in \Real^{n_1 \times \cdots \times n_d}$.
\end{enumerate}
\end{lemma}
\begin{proof}
    1. Let $\gamma_{\circ}^{\TT}(\Tensor{A}) = 0$ and take a minimizing sequence of TT cores $\{ \Tensor{G}_s^{(l)} \}_{s = 1}^{d}$ such that $\Tensor{A} = \TT(\Tensor{G}_1^{(l)}, \ldots, \Tensor{G}_d^{(l)})$ and $\prod_{s = 1}^{d} \Norm{\Tensor{G}_s^{(l)}}{\circ,\infty} \leq 2^{-l}$. For every entry of $\Tensor{A}$,
    \begin{equation*}
        |\Tensor{A}(i_1, \ldots, i_d)| \leq \prod\nolimits_{s = 1}^{d} \Norm{\Tensor{G}_s^{(l)}(:,i_s,:)}{\circ} \leq \prod\nolimits_{s = 1}^{d} \Norm{\Tensor{G}_s^{(l)}}{\circ,\infty} \leq 2^{-l}.
    \end{equation*}
    Hence, we get $\Tensor{A} = 0$ as $l \to \infty$. The other direction is obvious.

    2. If $c = 0$, the proof is finished. Otherwise, take a sequence of TT cores for $\Tensor{A}$ such that $\prod_{s = 1}^{d} \Norm{\Tensor{G}_s^{(l)}}{\circ,\infty} \leq \gamma_{\circ}^{\TT}(\Tensor{A}) + 2^{-l}$. Since $c\Tensor{A} = \TT(c \Tensor{G}_1^{(l)}, \Tensor{G}_2^{(l)}, \ldots, \Tensor{G}_d^{(l)})$, we get $\gamma_{\circ}^{\TT}(c\Tensor{A}) \leq |c| \prod_{s = 1}^{d} \Norm{\Tensor{G}_s^{(l)}}{\circ,\infty}$ so that $\gamma_{\circ}^{\TT}(c\Tensor{A}) \leq |c| \gamma_{\circ}^{\TT}(\Tensor{A})$ as $l \to \infty$. The inverse inequality is obtained similarly.

    3. Take two sequences of TT cores (for $\Tensor{A}$ and $\Tensor{B}$): $\prod_{s = 1}^{d} \Norm{\Tensor{G}_s^{(l)}}{\Frob,\infty} \leq \gamma_{\Frob}^{\TT}(\Tensor{A}) + 2^{-l}$ and $\prod_{s = 1}^{d} \Norm{\Tensor{H}_s^{(l)}}{\Frob,\infty} \leq \gamma_{\Frob}^{\TT}(\Tensor{B}) + 2^{-l}$. Then $\Tensor{A} + \Tensor{B} = \TT(\Tensor{F}_1^{(l)}, \ldots, \Tensor{F}_d^{(l)})$ with
    \begin{equation*}
        \Tensor{F}_s^{(l)}(:, i_s, :) =
        \begin{bmatrix}
              \Tensor{G}_s^{(l)}(:, i_s, :) & 0 \\
              0 & \Tensor{H}_s^{(l)}(:, i_s, :)
        \end{bmatrix}, \quad \Norm{\Tensor{F}_s^{(l)}}{\Frob,\infty}^2 \leq \Norm{\Tensor{G}_s^{(l)}}{\Frob,\infty}^2 + \Norm{\Tensor{H}_s^{(l)}}{\Frob,\infty}^2.
    \end{equation*}
    Without loss of generality, assume $\Norm{\Tensor{G}_1^{(l)}}{\Frob,\infty} = \cdots = \Norm{\Tensor{G}_d^{(l)}}{\Frob,\infty} = g_l$ and $\Norm{\Tensor{H}_1^{(l)}}{\Frob,\infty} = \cdots = \Norm{\Tensor{H}_d^{(l)}}{\Frob,\infty} = h_l$. Then $\prod_{s = 1}^{d} \Norm{\Tensor{F}_s^{(l)}}{\Frob,\infty} \leq (g_l^2 + h_l^2)^{d/2}$ and, as $l \to \infty$,
    \begin{equation*}
        \gamma_{\Frob}^{\TT}(\Tensor{A} + \Tensor{B}) \leq \big(\gamma_{\Frob}^{\TT}(\Tensor{A})^{\frac{2}{d}} + \gamma_{\Frob}^{\TT}(\Tensor{B})^{\frac{2}{d}}\big)^{\frac{d}{2}} = 2^{\frac{d}{2}} \big(\tfrac{1}{2} \gamma_{\Frob}^{\TT}(\Tensor{A})^{\frac{2}{d}} + \tfrac{1}{2} \gamma_{\Frob}^{\TT}(\Tensor{B})^{\frac{2}{d}}\big)^{\frac{d}{2}}.
    \end{equation*}
    It remains to use the convexity of $x \mapsto x^{d/2}$ for $d \geq 2$.

    4. Note that $\Norm{\Tensor{F}_s^{(l)}}{2,\infty} \leq \max\{\Norm{\Tensor{G}_s^{(l)}}{2,\infty}, \Norm{\Tensor{H}_s^{(l)}}{2,\infty}\}$.
\end{proof}

\subsection{Contributions}
We generalize \cref{theorem:matrix_approx} to the case of TT approximation and derive a componentwise bound on the TT rank.

\begin{theorem}
\label{theorem:tt_approx}
Let $\varepsilon \in (0,1)$ and $n_1, \ldots, n_d \in \N$. Consider
\begin{equation}
\label{eq:tt_rank}
    r = \left\lceil \frac{c_d}{\varepsilon^2} \log\left(2e \prod\nolimits_{s = 1}^{d} n_s \right) \right\rceil \in \N,
\end{equation}
where $c_d > 0$ is an absolute constant that depends only on $d$ and $e$ is Euler's number. For every $\Tensor{A} \in \Real^{n_1 \times \cdots \times n_d}$, there exists $\Tensor{B} \in \Real^{n_1 \times \cdots \times n_d}$ of $\ttrank{\Tensor{B}} \preccurlyeq r$ such that
\begin{equation*}
    \Norm{\Tensor{A} - \Tensor{B}}{\max} \leq \varepsilon \cdot \gamma_{\Frob}^{\TT}(\Tensor{A}).
\end{equation*}
\end{theorem}

To prove \cref{theorem:tt_approx}, we follow the technique developed in \cite{budzinskiy2024distance} for matrices and rely on the tensor-structured Hanson--Wright inequality of \cite{bamberger2022hanson}.

\begin{remark}
\cref{theorem:tt_approx} is formulated for real tensors because the results in \cite{bamberger2022hanson} are presented only in the real case. Their extension to complex tensors is beyond the scope of our article, but the complex variant of \cref{theorem:tt_approx} would follow immediately.
\end{remark}

As in \cref{theorem:matrix_approx}, the TT rank bound \cref{eq:tt_rank} scales as $\log(n)$ and suggests that the entrywise error decays like $r^{-1/2}$ as $r$ grows. The provable upper bound on $c_d$ grows at least as $d^d$, making \cref{theorem:tt_approx} primarily of theoretical interest and requiring $n$ to be large. The numerical results in \Cref{sec:numerical} for $d = 2,3$ suggest that $\gamma_{\Frob}^{\TT}$ has a higher impact on the quality of approximation than $c_d$ as $d$ increases.

We derive two corollaries from \cref{theorem:tt_approx} by estimating $\gamma_{\Frob}^{\TT}$ via specific TT factorizations. First, let $\Tensor{A}$ be represented in the CP format with factors $\Matrix{C}_s \in \Real^{n_s \times k}$,\begin{equation}
\label{eq:cp_format}
    \Tensor{A}(i_1, \ldots, i_d) = \sum\nolimits_{\alpha = 1}^{k} \Matrix{C}_1(i_1, \alpha) \Matrix{C}_2(i_2, \alpha) \ldots \Matrix{C}_d(i_d, \alpha),
\end{equation}
which is a specific type of the TT format (see \cref{remark:cp}). Denote by $\gamma^{\mathrm{CP}}$ the max-qnorm (the CP factorization quasinorm, in our nomenclature) introduced in \cite{ghadermarzy2019near}:
\begin{equation*}
\label{eq:gammacp}
    \gamma^{\mathrm{CP}}(\Tensor{A}) = \inf\set[\Big]{\prod\nolimits_{s = 1}^{d} \Norm{\Matrix{C}_s}{2,\infty}}{\{ \Matrix{C}_s \}_{s=1}^{d}~\text{such that}~\cref{eq:cp_format}}.
\end{equation*}

\begin{corollary}
\label{corollary:cp_approx}
Let $\varepsilon \in (0,1)$ and $n_1, \ldots, n_d \in \N$. Let $r \in \N$ be given by \cref{eq:tt_rank}. For every $\Tensor{A} \in \Real^{n_1 \times \cdots \times n_d}$, there exists $\Tensor{B} \in \Real^{n_1 \times \cdots \times n_d}$ of $\ttrank{\Tensor{B}} \preccurlyeq r$ such that
\begin{equation*}
    \Norm{\Tensor{A} - \Tensor{B}}{\max} \leq \varepsilon \cdot \gamma^{\mathrm{CP}}(\Tensor{A}).
\end{equation*}
\end{corollary}
\begin{proof}
    If \cref{eq:cp_format}, then $\Tensor{A} = \TT(\Tensor{G}_1, \ldots, \Tensor{G}_d)$ with TT cores $\Tensor{G}_1(1,i_1,:) = \Matrix{C}_1(i_1,:)$, $\Tensor{G}_d(:,i_d,1) = \Matrix{C}_d(i_d,:)^\trans$, $\Tensor{G}_s(:,i_s,:) = \mathrm{diag}[\Matrix{C}_s(i_s, :)]$. Then $\gamma_{\Frob}^{\TT}(\Tensor{A}) \leq \prod_{s = 1}^{d} \Norm{\Matrix{C}_s}{2,\infty}$.
\end{proof}

\begin{example}
\label{ex:eye_tensors}
Consider an order-$d$ diagonal tensor $\Tensor{I}_n \in \Real^{n \times \cdots \times n}$ with ones along the superdiagonal \cite{kolda2009tensor}. This tensor admits a CP representation with factors $\Matrix{C}_1 = \cdots = \Matrix{C}_d = \Matrix{I}_n \in \Real^{n \times n}$. Then \cref{corollary:cp_approx} guarantees that for any $\varepsilon \in (0,1)$ there exists a tensor $\Tensor{B} \in \Real^{n \times \cdots \times n}$ of TT rank bounded by \cref{eq:tt_rank} such that $\Norm{\Tensor{I}_n - \Tensor{B}_n}{\max} \leq \varepsilon$.
\end{example}

\begin{remark}
    Since the first version of this article appeared, \cref{corollary:cp_approx} has been applied to ``big-data'' tensors sampled from certain analytic functions \cite{budzinskiy2024big}.
\end{remark}

In the second \cref{corollary:tt_approx_coherence}, we consider a TT factorization obtained with the TT-SVD algorithm \cite{oseledets2011tensor} and estimate $\gamma_{\Frob}^{\TT}$ in terms of the unfoldings' spectral norm and the \textit{TT core coherences} \cite{budzinskiy2023tensor}. The latter are a generalization of the coherence used for subspaces and matrices \cite{candes2009exact}. A similar result for matrices appeared in \cite{budzinskiy2024distance}.

\subsection{Outline}
We prove \cref{theorem:tt_approx} in \Cref{sec:main_proof}. In \Cref{sec:core_coherence}, we introduce TT core coherences and prove \cref{corollary:tt_approx_coherence}. A heuristic algorithm of low-rank TT approximation in the maximum norm and the results of numerical experiments are presented in \Cref{sec:numerical}. A discussion in \Cref{sec:dicussion} concludes the article.

\subsection{Notation}
We use $\Matrix{a}$ for vectors, $\Matrix{A}$ for matrices, and $\Tensor{A}$ for tensors. We write $\mathrm{col}(\Matrix{A})$ and $\mathrm{row}(\Matrix{A})$ to denote the column and row spaces of $\Matrix{A}$. The standard basis vectors of $\Real^n$ are denoted by $\{ \Matrix{e}_i \}_{i = 1}^{n}$. The Kronecker product is written as $\otimes$ and the Kronecker delta as $\delta_{\alpha, \beta}$. We use $[n] = \{ 1, \ldots, n \}$ for $n \in \N$. For a tuple of numbers, $(r_1, \ldots, r_{n}) \preccurlyeq r$ means that $r_s \leq r$ for every $s \in [n]$. We denote the expectation of a random variable $\xi$ by $\Expectation \xi$, its $\mathrm{L_p}$-norm by $\Norm{\xi}{\Ell_p} = (\Expectation |\xi|^p)^{1/p}$, its sub-Gaussian norm by $\Norm{\xi}{\psi_2} = \sup_{p \geq 1} \{ \Norm{\xi}{\Ell_p} / \sqrt{p} \}$, and say that $\xi$ is a sub-Gaussian random variable if $\Norm{\xi}{\psi_2} < \infty$ \cite{vershynin2018high}. 
\section{Proof of Theorem~\ref{theorem:tt_approx}}
\label{sec:main_proof}
\subsection{Overview}
\label{sec:proof_overview}
Let $\Tensor{A} \in \Real^{n_1 \times \cdots \times n_d}$ and $r \in \N$ be given by \cref{eq:tt_rank}. It suffices to prove that for any TT factorization $\Tensor{A} = \TT(\Tensor{G}_1, \ldots, \Tensor{G}_d)$ there exists a tensor $\Tensor{B} \in \Real^{n_1 \times \cdots \times n_d}$ of $\ttrank{\Tensor{B}} \preccurlyeq r$ that satisfies
\begin{equation}
\label{eq:tt_bound_fixed_factorization}
    \Norm{\Tensor{A} - \Tensor{B}}{\max} \leq \varepsilon \cdot \prod\nolimits_{s = 1}^{d} \Norm{\Tensor{G}_s}{\Frob,\infty}.    
\end{equation}
Indeed, consider a minimizing sequence $\{ \Tensor{G}_s^{(l)} \}_{s = 1}^{d}$ such that $\Tensor{A} = \TT(\Tensor{G}_1^{(l)}, \ldots, \Tensor{G}_d^{(l)})$ and $\prod_{s = 1}^{d} \Norm{\Matrix{G}_s^{(l)}}{\Frob,\infty} \leq \gamma_{\Frob}^{\TT}(\Tensor{A}) + 2^{-l}$. It holds for the corresponding sequence of approximants $\Tensor{B}^{(l)} \in \Real^{n_1 \times \cdots \times n_d}$ with $\ttrank{\Tensor{B}^{(l)}} \preccurlyeq r$ that
\begin{equation*}
    \Norm{\Tensor{A} - \Tensor{B}^{(l)}}{\max} \leq \varepsilon \cdot \big( \gamma_{\Frob}^{\TT}(\Tensor{A}) + 2^{-l} \big).
\end{equation*}
The set $\set{\Tensor{B}}{\ttrank{\Tensor{B}} \preccurlyeq r}$ is closed\footnote{The $s$th component of $\ttrank{\Tensor{B}_\ast}$ is the rank of the $s$th unfolding matrix of $\Tensor{B}$ \cite{oseledets2011tensor,holtz2012manifolds}, so the set in question is the intersection of $d-1$ closed sets $\set{\Matrix{B}_s \in \Real^{(n_1 \ldots n_s) \times (n_{s+1} \ldots n_d)}}{\rank{\Matrix{B}_s} \leq r}$.}. Since $\{ \Tensor{B}^{(l)} \}$ is bounded, it has a limit point $\Tensor{B}_\ast$ such that $\ttrank{\Tensor{B}_\ast} \preccurlyeq r$ and $\Norm{\Tensor{A} - \Tensor{B}_\ast}{\max} \leq \varepsilon \cdot \gamma_{\Frob}^{\TT}(\Tensor{A})$---as desired.

To address \cref{eq:tt_bound_fixed_factorization}, let the TT rank of the selected factorization be $(k_1, \ldots, k_{d-1})$ and draw random matrices $\{ \Matrix{R}_s \}_{s = 1}^{d-1}$ of size $\Matrix{R}_s \in \Real^{k_s \times r}$ with independent and identically distributed sub-Gaussian entries \cite{vershynin2018high}. We intend to approximate $\Tensor{A}$ as
\begin{equation}
\label{eq:tt_random_embeddings}
    \Tensor{A}(i_1, \ldots, i_d) \approx \Tensor{G}_1(1, i_1, :) \Matrix{R}_1 \Matrix{R}_1^\trans \Tensor{G}_2(:, i_2, :) \Matrix{R}_2 \Matrix{R}_2^\trans \ldots \Matrix{R}_{d-1} \Matrix{R}_{d-1}^\trans \Tensor{G}_d(:, i_d, 1).
\end{equation}
Using concentration inequalities, we will show that for every individual entry of $\Tensor{A}$ the approximation error in \cref{eq:tt_random_embeddings} is bounded by $\varepsilon \prod_{s = 1}^{d} \Norm{\Tensor{G}_s(:, i_s, :)}{\Frob}$ with high probability. A union bound then shows that such error bounds hold for all entries simultaneously with constant probability, hence the sought matrices $\{ \Matrix{R}_s \}_{s = 1}^{d-1}$, and tensor $\Tensor{B}$, exist.

The right-hand side in \cref{eq:tt_random_embeddings}, as a function of random variables, belongs to a family\begin{equation}
\label{eq:quadratic_form}
    \Delta(\Matrix{R}_1, \ldots, \Matrix{R}_{d-1}) = (\Matrix{W}_1 \Matrix{R}_1) (\Matrix{R}_1^\trans \Matrix{W}_2 \Matrix{R}_2) \ldots (\Matrix{R}_{d-2}^\trans \Matrix{W}_{d-1} \Matrix{R}_{d-1}) (\Matrix{R}_{d-1}^\trans \Matrix{W}_{d})
\end{equation}
of quadratic forms, parametrized by $\Matrix{W}_1 \in \Real^{1 \times k_1}$, $\Matrix{W}_s \in \Real^{k_{s-1} \times k_s}$, $\Matrix{W}_d \in \Real^{k_{d-1} \times 1}$. In the rest of this section, we analyze \cref{eq:quadratic_form} in order to prove \cref{eq:tt_bound_fixed_factorization}.

\subsection{Arrays and indices}
Here, we review the notation used in \cite{bamberger2022hanson}. Consider a tuple of dimensions\footnote{We will substitute $N = d-1$ later, but will use $N$ now for generality and to reduce clutter.} $\Pind{m} = (m_1, \ldots, m_N)$. We will extensively work with subsets of dimensions as selected by $\Omega \subseteq [N]$. A map $\Pind{i} : \Omega \to \N$ such that $\Pind{i}_{\omega} \in [m_\omega]$ for every $\omega \in \Omega$ is called a partial index on $\Omega$. We denote the set of all such partial indices by $\PindSet^{\Pind{m}}(\Omega)$. Partial indices on $[N]$ are called indices, and we write $\PindSet^{\Pind{m}} = \PindSet^{\Pind{m}}([N])$.

Maps $\Tensor{C} : \PindSet^{\Pind{m}} \to \Real$ and $\Tensor{D} : \PindSet^{\Pind{m}}(\Omega) \to \Real$ are called an array and a partial array. We denote the set of all arrays by $\Real^{\Pind{m}}$ and of all partial arrays by $\Real^{\Pind{m}}(\Omega)$. We use the same font for tensors and arrays, but write $\Tensor{D}_{\Pind{i}}$ to access the entries of an array.

For $\Pind{i} \in \PindSet^{\Pind{m}}(\Omega)$ and $\Omega' \subset \Omega$, we denote the restriction of $\Pind{i}$ to $\Omega'$ by $\Pind{i}_{\Omega'} \in \PindSet^{\Pind{m}}(\Omega')$ so that $(\Pind{i}_{\Omega'})_{\omega} = \Pind{i}_{\omega}$ for $\omega \in \Omega'$. If $\Omega_1, \Omega_2 \subset [N]$ are disjoint, two partial indices $\Pind{i} \in \PindSet^{\Pind{m}}(\Omega_1)$ and $\Pind{j} \in \PindSet^{\Pind{m}}(\Omega_2)$ can be joined into $\Pind{i} \PindJoin \Pind{j} \in \PindSet^{\Pind{m}}(\Omega_1 \cup \Omega_2)$ so that
\begin{equation*}
    (\Pind{i} \PindJoin \Pind{j})_{\omega} = \begin{cases}
        \Pind{i}_{\omega}, & \omega \in \Omega_1, \\
        \Pind{j}_{\omega}, & \omega \in \Omega_2.
    \end{cases}
\end{equation*}

Nonempty $\{ \Omega_\tau \}_{\tau = 1}^{\kappa} \subseteq \Omega$ are a cover of $\Omega$ if $\Omega = \cup_{\tau = 1}^{\kappa} \Omega_\tau$. If a cover consists of mutually disjoint subsets, we call it a partition of $\Omega$ and say that each $\Omega_\tau$ is a cell. We denote by $\Part_{\kappa}(\Omega)$ the collection of all partitions of $\Omega$ into $\kappa$ cells.

\begin{lemma}
\label{lemma:associative_join}
    Let $\Omega \subseteq [N]$ and $\{ \Omega_\tau \}_{\tau = 1}^{\kappa} \in \Part_{\kappa}(\Omega)$. For every $\Pind{i} \in \PindSet^{\Pind{m}}(\Omega)$, the map $\PindJoin$ is associative and commutative when it acts on $\{ \Pind{i}_{\Omega_\tau} \}_{\tau = 1}^{\kappa}$ and $\Pind{i} = \Pind{i}_{\Omega_1} \PindJoin \cdots \PindJoin \Pind{i}_{\Omega_\kappa}$.
\end{lemma}
\begin{proof}
    For $\kappa = 2$, $\Pind{i}_{\Omega_1} \PindJoin \Pind{i}_{\Omega_2} = \Pind{i}_{\Omega_2} \PindJoin \Pind{i}_{\Omega_1}$ by definition. For $\kappa = 3$, let $\Pind{j} = (\Pind{i}_{\Omega_1} \PindJoin \Pind{i}_{\Omega_2}) \PindJoin \Pind{i}_{\Omega_3}$. No matter the ordering of the cells, we have
    \begin{equation*}
        \Pind{j}_{\omega} = \begin{cases}
        (\Pind{i}_{\Omega_1} \PindJoin \Pind{i}_{\Omega_2})_{\omega}, & \omega \in \Omega_1 \cup \Omega_2, \\
        (\Pind{i}_{\Omega_3})_{\omega}, & \omega \in \Omega_3,
    \end{cases}
    =
    \begin{cases}
        (\Pind{i}_{\Omega_1})_{\omega}, & \omega \in \Omega_1, \\
        (\Pind{i}_{\Omega_2})_{\omega}, & \omega \in \Omega_2, \\
        (\Pind{i}_{\Omega_3})_{\omega}, & \omega \in \Omega_3.
    \end{cases}
    \end{equation*}
    The rest follows by induction on $\kappa$.
\end{proof}

Let $\Pind{m}^{\times 2} = (m_1, \ldots, m_N, m_1, \ldots, m_N)$. Given $\Omega', \Omega'' \subseteq [N]$, we define
\begin{equation*}
    \Omega' \oplus \Omega'' = \Omega' \cup (\Omega'' + N) \subseteq [2N],
\end{equation*}
where $N$ is added to every element of $\Omega''$. We can concatenate two partial indices $\Pind{i} \in \PindSet^{\Pind{m}}(\Omega')$ and $\Pind{j} \in \PindSet^{\Pind{m}}(\Omega'')$ into $\Pind{i} \PindCat \Pind{j} \in \PindSet^{\Pind{m}^{\times 2}}(\Omega' \oplus \Omega'')$ according to
\begin{equation*}
    (\Pind{i} \PindCat \Pind{j})_{\omega} = \begin{cases}
        \Pind{i}_{\omega}, & \omega \in \Omega', \\
        \Pind{j}_{\omega - N}, & \omega \in \Omega'' + N.
    \end{cases}
\end{equation*}

The arrays corresponding to the quadratic forms \cref{eq:quadratic_form} exhibit more inherent structure than the general arrays studied in \cite{bamberger2022hanson}. To handle this structure, we introduce additional index notation. Consider tuples $\Pind{r} = (r_1, \ldots, r_N)$ and $\Pind{k} = (k_1, \ldots, k_N)$. We define their product as $\PindKron{\Pind{r}}{\Pind{k}} = (r_1 k_1, \ldots, r_N k_N)$. It follows from the definition that $\PindKron{\Pind{r}}{\Pind{k}}^{\times 2} = \PindKron{\Pind{r}^{\times 2}}{\Pind{k}^{\times 2}}$. We define the product of $\Pind{\alpha} \in \PindSet^{\Pind{r}}(\Omega)$ and $\Pind{i} \in \PindSet^{\Pind{k}}(\Omega)$ as $\PindKron{\Pind{\alpha}}{\Pind{i}} \in \PindSet^{\PindKron{\Pind{r}}{\Pind{k}}}(\Omega)$:
\begin{equation*}
    (\PindKron{\Pind{\alpha}}{\Pind{i}})_{\omega} = \Pind{i}_{\omega} + (\Pind{\alpha}_\omega - 1) k_\omega, \quad \omega \in \Omega.
\end{equation*}

\begin{lemma}
\label{lemma:kron_join}
    Let $\Omega_1, \Omega_2 \subset [N]$ be disjoint. For all $\Pind{\alpha} \in \PindSet^{\Pind{r}}(\Omega_1)$, $\Pind{i} \in \PindSet^{\Pind{k}}(\Omega_1)$, $\Pind{\beta} \in \PindSet^{\Pind{r}}(\Omega_2)$, and $\Pind{j} \in \PindSet^{\Pind{k}}(\Omega_2)$, it holds that
    \begin{equation*}
        \PindKron{\Pind{\alpha}}{\Pind{i}} \PindJoin \PindKron{\Pind{\beta}}{\Pind{j}} = \PindKron{(\Pind{\alpha} \PindJoin \Pind{\beta})}{(\Pind{i} \PindJoin \Pind{j})} \in \PindSet^{\PindKron{\Pind{r}}{\Pind{k}}}(\Omega_1 \cup \Omega_2).
    \end{equation*}
\end{lemma}
\begin{proof}
Let $\Pind{\ell} = \PindKron{(\Pind{\alpha} \PindJoin \Pind{\beta})}{(\Pind{i} \PindJoin \Pind{j})}$. For every $\omega \in \Omega_1 \cup \Omega_2$, we have
\begin{equation*}
    \Pind{\ell}_{\omega} = (\Pind{i} \PindJoin \Pind{j})_\omega + ((\Pind{\alpha} \PindJoin \Pind{\beta})_{\omega} - 1) k_\omega = \begin{cases}
        \Pind{i}_\omega + (\Pind{\alpha}_\omega - 1) k_\omega = (\PindKron{\Pind{\alpha}}{\Pind{i}})_{\omega}, & \omega \in \Omega_1, \\
        \Pind{j}_\omega + (\Pind{\beta}_\omega - 1) k_\omega = (\PindKron{\Pind{\beta}}{\Pind{j}})_{\omega}, & \omega \in \Omega_2.
    \end{cases}
\end{equation*}
\end{proof}

\begin{lemma}
\label{lemma:kron_cat}
    Let $\Omega', \Omega'' \subseteq [N]$. For all $\Pind{\alpha} \in \PindSet^{\Pind{r}}(\Omega')$, $\Pind{i} \in \PindSet^{\Pind{k}}(\Omega')$, $\Pind{\beta} \in \PindSet^{\Pind{r}}(\Omega'')$, and $\Pind{j} \in \PindSet^{\Pind{k}}(\Omega'')$, it holds that
    \begin{equation*}
        \PindKron{\Pind{\alpha}}{\Pind{i}} \PindCat \PindKron{\Pind{\beta}}{\Pind{j}} = \PindKron{(\Pind{\alpha} \PindCat \Pind{\beta})}{(\Pind{i} \PindCat \Pind{j})} \in \PindSet^{\PindKron{\Pind{r}}{\Pind{k}}^{\times 2}}(\Omega' \oplus \Omega'').
    \end{equation*}
\end{lemma}
\begin{proof}
Let $\Pind{\ell} = \PindKron{(\Pind{\alpha} \PindCat \Pind{\beta})}{(\Pind{i} \PindCat \Pind{j})}$. For every $\omega \in \Omega'$, we have
\begin{equation*}
    \Pind{\ell}_{\omega} = (\Pind{i} \PindCat \Pind{j})_\omega + ((\Pind{\alpha} \PindCat \Pind{\beta})_{\omega} - 1) k_\omega = \Pind{i}_\omega + (\Pind{\alpha}_\omega - 1) k_\omega = (\PindKron{\Pind{\alpha}}{\Pind{i}})_{\omega},
\end{equation*}
which is equal to $(\PindKron{\Pind{\alpha}}{\Pind{i}} \PindCat \PindKron{\Pind{\beta}}{\Pind{j}})_\omega$. Similarly, for every $\omega \in \Omega'' + N$,
\begin{equation*}
    \Pind{\ell}_{\omega} = (\Pind{i} \PindCat \Pind{j})_\omega + ((\Pind{\alpha} \PindCat \Pind{\beta})_{\omega} - 1) k_{\omega-N} = \Pind{j}_{\omega-N} + (\Pind{\beta}_{\omega-N} - 1) k_{\omega-N} = (\PindKron{\Pind{\beta}}{\Pind{j}})_{\omega - N}.
\end{equation*}
\end{proof}

These index operations are introduced to facilitate access to the entries of structured arrays: $\Pind{i} \PindCat \Pind{j}$ for $\Real^{\Pind{m}^{\times 2}}(\Omega' \oplus \Omega'')$ and $\PindKron{\Pind{\alpha}}{\Pind{i}} \PindCat \PindKron{\Pind{\beta}}{\Pind{j}}$ for $\Real^{\PindKron{\Pind{r}}{\Pind{k}}^{\times 2}}(\Omega' \oplus \Omega'')$.

\subsection{Array transformations}
Let $\Omega' \subseteq \Omega \subseteq [N]$ and $\Tensor{C} \in \Real^{\Pind{m}^{\times 2}}(\Omega \oplus \Omega)$. Following \cite{bamberger2022hanson}, we define\footnote{We use $\langle \cdot \rangle$ instead of $(\cdot)$ \cite{bamberger2022hanson} to avoid possible confusion with the notation for collections.} $\Tensor{C}^{\langle\Omega'\rangle} \in \Real^{\Pind{m}^{\times 2}}((\Omega \setminus \Omega') \oplus (\Omega \setminus \Omega'))$ as
\begin{equation*}
    \Tensor{C}^{\langle\Omega'\rangle}_{\Pind{i} \PindCat \Pind{j}} = \sum\nolimits_{\Pind{\ell} \in \PindSet^{\Pind{m}}(\Omega')} \Tensor{C}_{(\Pind{i} \PindJoin \Pind{\ell}) \PindCat (\Pind{j} \PindJoin \Pind{\ell})}, \quad \Pind{i}, \Pind{j} \in \PindSet^{\Pind{m}}(\Omega \setminus \Omega').
\end{equation*}
We refer to $\Tensor{C}^{\langle\Omega'\rangle}$ as the partial trace: when $\Omega' = \Omega$, the array $\Tensor{C}^{\langle\Omega'\rangle}$ is a real number obtained by summing the ``diagonal'' entries of $\Tensor{C}$.

For $\Tensor{D} \in \Real^{\Pind{m}}(\Omega)$, we define its partial Frobenius norm $\Norm{\Tensor{D}}{\Frob(\Omega')} \in \Real^{\Pind{m}}(\Omega \setminus \Omega')$ as
\begin{equation*}
    (\Norm{\Tensor{D}}{\Frob(\Omega')})_{\Pind{i}} = \Big(\sum\nolimits_{\Pind{j} \in \PindSet^{\Pind{m}}(\Omega')} |\Tensor{D}_{\Pind{i} \PindJoin \Pind{j}}|^2 \Big)^{\frac{1}{2}}, \quad \Pind{i} \in \PindSet^{\Pind{m}}(\Omega \setminus \Omega').
\end{equation*}
When $\Omega' = \Omega$, we write $\Norm{\Tensor{D}}{\Frob} = \Norm{\Tensor{D}}{\Frob(\Omega')}$, which is the Frobenius norm\footnote{The set of partial arrays $\Real^{\Pind{m}}(\Omega)$ is a linear space with operations defined just as for tensors.} on $\Real^{\Pind{m}}(\Omega)$.

\begin{lemma}
\label{lemma:associative_trace_frob}
    Let $\Omega_1, \Omega_2 \subset \Omega \subseteq [N]$ be disjoint. Then
    \begin{enumerate}
        \item $[\Tensor{C}^{\langle\Omega_1\rangle}]^{\langle\Omega_2\rangle} = [\Tensor{C}^{\langle\Omega_2\rangle}]^{\langle\Omega_1\rangle} = \Tensor{C}^{\langle\Omega_1 \cup \Omega_2\rangle}$ for all $\Tensor{C} \in \Real^{\Pind{m}^{\times 2}}(\Omega \oplus \Omega)$,
        \item $\Norm{\Norm{\Tensor{D}}{\Frob(\Omega_1)}}{\Frob(\Omega_2)} = \Norm{\Norm{\Tensor{D}}{\Frob(\Omega_2)}}{\Frob(\Omega_1)} = \Norm{\Tensor{D}}{\Frob(\Omega_1 \cup \Omega_2)}$ for all $\Tensor{D} \in \Real^{\Pind{m}}(\Omega)$.
    \end{enumerate}
\end{lemma}
\begin{proof}
Denote $\Tensor{\widehat{C}} = \Tensor{C}^{\langle\Omega_1\rangle}$. For any $\Pind{i}, \Pind{j} \in \PindSet^{\Pind{m}}((\Omega \setminus \Omega_1) \setminus \Omega_2) = \PindSet^{\Pind{m}}(\Omega \setminus (\Omega_1 \cup \Omega_2))$, 
\begin{align*}
    \Tensor{\widehat{C}}^{\langle\Omega_2\rangle}_{\Pind{i} \PindCat \Pind{j}} &= \sum_{\Pind{\ell} \in \PindSet^{\Pind{m}}(\Omega_2)} \widehat{\Tensor{C}}_{(\Pind{i} \PindJoin \Pind{\ell}) \PindCat (\Pind{j} \PindJoin \Pind{\ell})} = \sum_{\Pind{\ell} \in \PindSet^{\Pind{m}}(\Omega_2), \Pind{\ell'} \in \PindSet^{\Pind{m}}(\Omega_1)} \Tensor{C}_{(\Pind{i} \PindJoin \Pind{\ell} \PindJoin \Pind{\ell'}) \PindCat (\Pind{j} \PindJoin \Pind{\ell} \PindJoin \Pind{\ell'})} \\
    &= \sum_{\Pind{\ell''} \in \PindSet^{\Pind{m}}(\Omega_1 \cup \Omega_2)} \Tensor{C}_{(\Pind{i} \PindJoin \Pind{\ell''}) \PindCat (\Pind{j} \PindJoin \Pind{\ell''})} = \Tensor{C}^{\langle\Omega_1 \cup \Omega_2\rangle}_{\Pind{i} \PindCat \Pind{j}}.
\end{align*}
We used \cref{lemma:associative_join} in the second and third equalities. Reversing the roles of $\Omega_1$ and $\Omega_2$, we prove the first part. The proof of the second part is similar.
\end{proof}

Given $\Tensor{E} \in \Real^{\PindKron{\Pind{r}}{\Pind{k}}}(\Omega)$ and $\Pind{\alpha} \in \PindSet^{\Pind{r}}(\Omega)$, we can extract a subarray $\Tensor{E}^{\{ \Pind{\alpha} \}} \in \Real^{\Pind{k}}(\Omega)$ such that $\Tensor{E}^{\{ \Pind{\alpha} \}}_{\Pind{i}} = \Tensor{E}_{\PindKron{\Pind{\alpha}}{\Pind{i}}}$ for $\Pind{i} \in \PindSet^{\Pind{k}}(\Omega)$. By \cref{lemma:kron_cat}, if $\Tensor{F} \in \Real^{\PindKron{\Pind{r}}{\Pind{k}}^{\times 2}}(\Omega' \oplus \Omega'')$, then $\Tensor{F}^{\{ \Pind{\alpha} \PindCat \Pind{\beta} \}} \in \Real^{\Pind{k}^{\times 2}}(\Omega' \oplus \Omega'')$ and $\Tensor{F}^{\{ \Pind{\alpha} \PindCat \Pind{\beta} \}}_{\Pind{i} \PindCat \Pind{j}} = \Tensor{F}_{\PindKron{\Pind{\alpha}}{\Pind{i}} \PindCat \PindKron{\Pind{\beta}}{\Pind{j}}}$ for $\Pind{\alpha} \in \PindSet^{\Pind{r}}(\Omega')$ and $\Pind{\beta} \in \PindSet^{\Pind{r}}(\Omega'')$. In addition, if $\Tensor{G} \in \Real^{\PindKron{\Pind{r}}{\Pind{k}}^{\times 2}}(\mho)$ for $\mho \subseteq \Omega' \oplus \Omega''$, we shall write $\Tensor{G}^{\{ \Pind{\alpha} \PindCat \Pind{\beta} \}_{\mho}} = \Tensor{G}^{\{ (\Pind{\alpha} \PindCat \Pind{\beta})_{\mho} \}}$ to extract a subarray of $\Tensor{G}$ based on the partial index $\Pind{\alpha} \PindCat \Pind{\beta}$.

\subsection{Hierarchy of norms}

We can associate a norm on $\Real^{\Pind{m}}(\Omega)$ with every partition of dimensions $\pi \in \Part_{\kappa}(\Omega)$. Define
\begin{equation*}
    \Norm{\Tensor{C}}{\pi} = \sup \set[\Bigg]{\sum_{\Pind{i} \in \PindSet^{\Pind{m}}(\Omega)} \Tensor{C}_{\Pind{i}} \prod_{\tau = 1}^{\kappa} \Tensor{Z}^{(\tau)}_{\Pind{i}_{\Omega_\tau}}}{\Tensor{Z}^{(\tau)} \in \Real^{\Pind{m}}(\Omega_\tau),~\Norm{\Tensor{Z}^{(\tau)}}{\Frob} = 1}.
\end{equation*}
This family of norms can be partially ordered according to the partition lattice based on whether one partition is a refinement of another \cite{wang2017operator}. For instance,
\begin{equation}
\label{eq:norm_hierarchy}
    \max_{\kappa \in [|\Omega|]} \max_{\pi \in \Part_{\kappa}(\Omega)} \Norm{\Tensor{C}}{\pi} = \Norm{\Tensor{C}}{\{ \Omega \}} = \Norm{\Tensor{C}}{\Frob}.
\end{equation}

\subsection{Tensor-structured Hanson--Wright inequality}
The main result of \cite{bamberger2022hanson}, and the main instrument we need to prove \cref{theorem:tt_approx}, is an upper bound on the moments of a quadratic form in sub-Gaussian random variables. The following theorem is a slight modification of \cite[Theorem~3]{bamberger2022hanson}. We write $\Compl{\Omega} = [N] \setminus \Omega$ for $\Omega \subseteq [N]$.

\begin{theorem}
\label{theorem:moments_bound}
    Let $N \geq 1$, $\Pind{m} \in \N^N$, and $\Tensor{C} \in \Real^{\Pind{m}}$. Consider a quadratic form $f_{\Tensor{C}} : \Real^{m_1} \times \cdots \times \Real^{m_N} \to \Real$ defined by
    \begin{equation*}
        f_{\Tensor{C}}(\Matrix{x}_1, \ldots, \Matrix{x}_N) = \sum_{\Pind{i}, \Pind{j} \in \PindSet^{\Pind{m}}} \Tensor{C}_{\Pind{i} \PindCat \Pind{j}} \prod_{s = 1}^{N} \Matrix{x}_s(\Pind{i}_s) \Matrix{x}_s(\Pind{j}_s).
    \end{equation*}
    Let $\{ \Matrix{x}_s \}_{s = 1}^{N}$ be random vectors with independent sub-Gaussian entries satisfying
    \begin{equation*}
        \Expectation \Matrix{x}_s(\Pind{i}_s) = 0, \quad \Expectation |\Matrix{x}_s(\Pind{i}_s)|^2 = \eta_s^2, \quad \Norm{\Matrix{x}_s(\Pind{i}_s)}{\psi_2} \leq \eta_s L
    \end{equation*}
    for some constants $\eta_s > 0$ and $L \geq 1$ and all $\Pind{i} \in \PindSet^{\Pind{m}}$. Then there exists an absolute constant\footnote{The value of $\mathring{c}_{N}$ can be traced to \cite{latala2006estimates} from where we can deduce that it grows at least as $N^N$.} $\mathring{c}_{N} > 0$ that depends only on $N$ such that for all $p \geq 1$ it holds that
    \begin{equation*}
        \Norm{f_{\Tensor{C}} - \Expectation f_{\Tensor{C}}}{\Ell_p} \leq \mathring{c}_{N} \prod_{s = 1}^{N} (\eta_s L)^2 \sum_{\varnothing \neq \Omega \subseteq [N]} \sum_{\kappa = 1}^{2|\Omega|} p^{\frac{\kappa}{2}} \sum_{\pi \in \Part_{\kappa} (\Omega \oplus \Omega)} \Norm{\Tensor{C}^{\langle\Compl{\Omega}\rangle}}{\pi}.
    \end{equation*}
\end{theorem}
\begin{proof}
    Consider scaled random vectors $\{ \Matrix{y}_s \}_{s = 1}^{N}$ defined as $\Matrix{y}_s = \frac{1}{\eta_s} \Matrix{x}_s$. The variance of each of their entries is equal to one, so we can apply \cite[Theorem~3]{bamberger2022hanson} to the quadratic form $f_{\Tensor{C}}$ in the $\Matrix{y}$-variables and note that the function $f_{\Tensor{C}}$ is homogeneous of order two in each variable. We also reverse the roles of $\Omega$ and $\Compl{\Omega}$.
\end{proof}

To apply \cref{theorem:moments_bound} to a specific quadratic form, one needs to write down the array $\Tensor{C}$ of the quadratic form and derive an upper bound on $\Norm{\Tensor{C}^{(\Compl{\Omega})}}{\pi}$.

\subsection{Technical lemmas}

Before we turn to the analysis of the quadratic form \cref{eq:quadratic_form}, we prove two lemmas that describe a variant of the Cauchy--Schwarz inequality.

\begin{lemma}
\label{lemma:cauchy_schwarz}
Let $\varnothing \neq \Omega \subseteq [N]$ and $\pi = \{ \mho_\tau \}_{\tau = 1}^{\kappa} \in \Part_\kappa (\Omega \oplus \Omega)$. For $\omega \in \Omega$, define 
    \begin{equation*}
        \theta_{\pi, \omega} = \begin{cases}
            1/2 & \text{if}~\omega~\text{and}~\omega+N~\text{belong to the same cell of}~\pi, \\
            0   & \text{otherwise}. \\
        \end{cases}
    \end{equation*} 
    For every $\tau \in [\kappa]$, let $\Omega_\tau = (\mho_\tau \cup (\mho_\tau - N)) \cap \Omega$. Then for every $\Pind{m} \in \N^N$ and every collection of partial arrays $\{ \Tensor{C}^{(\tau)} \}_{\tau = 1}^{\kappa}$ such that $\Tensor{C}^{(\tau)} \in \Real^{\Pind{m}}(\Omega_\tau)$ it holds that
    \begin{equation*}
        \sum_{\Pind{i} \in \PindSet^{\Pind{m}}(\Omega)} \prod_{\tau \in [\kappa]} \Tensor{C}^{(\tau)}_{\Pind{i}_{\Omega_\tau}} \leq \prod_{\omega \in \Omega} m_\omega^{\theta_{\pi, \omega}} \prod_{\tau \in [\kappa]} \Norm{\Tensor{C}^{(\tau)}}{\Frob}.
    \end{equation*}
\end{lemma}
\begin{proof}
    We will prove the lemma by induction over the cardinality $|\Omega|$. For the base case, let $\Omega = \{ \omega \}$ be a singleton. There are two partitions of $\Omega \oplus \Omega$:
    \begin{equation*}
        \pi_1 = \Big\{ \big\{ \omega, \omega+N \big\} \Big\}, \quad \pi_2 = \Big\{ \big\{ \omega \big\}, \big\{ \omega + N \big\} \Big\}.
    \end{equation*}
    Fix $\pi_1$. Then $\kappa = 1$ and $\Omega_1 = \Omega = \{ \omega \}$. By the Cauchy--Schwarz inequality,
    \begin{equation*}
        \sum_{\Pind{i} \in \PindSet^{\Pind{m}}(\Omega)} \prod_{\tau \in [\kappa]} \Tensor{C}^{(\tau)}_{\Pind{i}_{\Omega_\tau}} = \sum_{\Pind{i} \in \PindSet^{\Pind{m}}(\Omega)} \Tensor{C}^{(1)}_{\Pind{i}} \leq \sqrt{m_\omega} \cdot \Norm{\Tensor{C}^{(1)}}{\Frob} = m_\omega^{\theta_{\pi_1, \omega}} \cdot \Norm{\Tensor{C}^{(1)}}{\Frob}.
    \end{equation*}
    Now, fix $\pi_2$. Then $\kappa = 2$ and $\Omega_1 = \Omega_2 = \Omega = \{ \omega \}$. Similarly, we obtain
    \begin{equation*}
        \sum_{\Pind{i} \in \PindSet^{\Pind{m}}(\Omega)} \prod_{\tau \in [\kappa]} \Tensor{C}^{(\tau)}_{\Pind{i}_{\Omega_\tau}} = \sum_{\Pind{i} \in \PindSet^{\Pind{m}}(\Omega)} \Tensor{C}^{(1)}_{\Pind{i}} \Tensor{C}^{(2)}_{\Pind{i}} \leq \Norm{\Tensor{C}^{(1)}}{\Frob} \Norm{\Tensor{C}^{(2)}}{\Frob} = m_\omega^{\theta_{\pi_2, \omega}} \cdot \Norm{\Tensor{C}^{(1)}}{\Frob} \Norm{\Tensor{C}^{(2)}}{\Frob}.
    \end{equation*}

    Next, let $|\Omega| \geq 2$ and assume that the lemma holds for all subsets of at most $|\Omega| - 1$ dimensions. Fix $\pi = \{ \mho_\tau \}_{\tau = 1}^{\kappa} \in \Part_\kappa (\Omega \oplus \Omega)$. Let $\omega \in \Omega$ be such that $\theta_{\pi, \omega} = 1/2$, that is, $\omega$ lies in exactly one $\Omega_t = (\mho_t \cup (\mho_t - N)) \cap \Omega$. Consider
    \begin{equation*}
        \sum_{\Pind{i} \in \PindSet^{\Pind{m}}(\Omega)} \prod_{\tau \in [\kappa]} \Tensor{C}^{(\tau)}_{\Pind{i}_{\Omega_\tau}} = \sum_{\Pind{j} \in \PindSet^{\Pind{m}}(\Omega \setminus \{ \omega \})} \bigg( \sum_{\Pind{\ell} \in \PindSet^{\Pind{m}}(\{ \omega \})} \Tensor{C}^{(t)}_{\Pind{\ell} \PindJoin \Pind{j}_{\Omega_t}} \bigg) \prod_{\tau \neq t} \Tensor{C}^{(\tau)}_{\Pind{j}_{\Omega_\tau}}.
    \end{equation*}
    Let $\Tensor{D} = \Norm{\Tensor{C}^{(t)}}{\Frob(\{\omega\})} \in \Real^{\Pind{m}}(\Omega_t \setminus \{\omega\})$. For every fixed $\Pind{j}$, the Frobenius norm of the partial array $\Pind{\ell} \mapsto \Tensor{C}^{(t)}_{\Pind{\ell} \PindJoin \Pind{j}_{\Omega_t}}$ is equal to $\Tensor{D}_{\Pind{j}_{\Omega_t}}$. Then, by the Cauchy--Schwarz inequality, 
    \begin{equation*}
        \sum_{\Pind{j} \in \PindSet^{\Pind{m}}(\Omega \setminus \{ \omega \})} \bigg( \sum_{\Pind{\ell} \in \PindSet^{\Pind{m}}(\{ \omega \})} \Tensor{C}^{(t)}_{\Pind{\ell} \PindJoin \Pind{j}_{\Omega_t}} \bigg) \prod_{\tau \neq t} \Tensor{C}^{(\tau)}_{\Pind{j}_{\Omega_\tau}} \leq \sum_{\Pind{j} \in \PindSet^{\Pind{m}}(\Omega \setminus \{ \omega \})} \big( \sqrt{m_\omega} \cdot \Tensor{D}_{\Pind{j}_{\Omega_t}} \big) \prod_{\tau \neq t} \Tensor{C}^{(\tau)}_{\Pind{j}_{\Omega_\tau}}.
    \end{equation*}
    This removes the presence of $\omega$ and reduces the problem to the case of $\Omega' = \Omega \setminus \{ \omega \}$. We can modify $\pi$ to construct a new partition
    \begin{equation*}
        \pi' = (\pi \setminus \mho_t) \cup (\mho_t \setminus \{ \omega, \omega+N \}) \in \begin{cases}
            \Part_{\kappa - 1}(\Omega' \oplus \Omega'), & \mho_t \setminus \big\{ \omega, \omega+N \big\} = \varnothing, \\
            \Part_{\kappa}(\Omega' \oplus \Omega'), & \mho_t \setminus \big\{ \omega, \omega+N \big\} \neq \varnothing.
        \end{cases}
    \end{equation*}
    Note that $\theta_{\pi,\omega'} = \theta_{\pi', \omega'}$ for $\omega' \in \Omega'$. By the induction hypothesis and \cref{lemma:associative_trace_frob},
    \begin{align*}
        \sum_{\Pind{i} \in \PindSet^{\Pind{m}}(\Omega)} \prod_{\tau \in [\kappa]} \Tensor{C}^{(\tau)}_{\Pind{i}_{\Omega_\tau}} &\leq m_\omega^{\theta_{\pi, \omega}} \sum_{\Pind{j} \in \PindSet^{\Pind{m}}(\Omega \setminus \{ \omega \})} \Tensor{D}_{\Pind{j}_{\Omega_t}} \prod_{\tau \neq t} \Tensor{C}^{(\tau)}_{\Pind{j}_{\Omega_\tau}} \\
        &\leq m_\omega^{\theta_{\pi, \omega}} \prod_{\omega' \in \Omega'} m_{\omega'}^{\theta_{\pi', \omega'}} \Norm{\Tensor{D}}{\Frob} \prod_{\tau \neq t} \Norm{\Tensor{C}^{(\tau)}}{\Frob} = \prod_{\omega \in \Omega} m_\omega^{\theta_{\pi, \omega}} \prod_{\tau \in [\kappa]} \Norm{\Tensor{C}^{(\tau)}}{\Frob}.
    \end{align*}

    Suppose now that $\theta_{\pi, \omega} = 0$, that is, $\omega$ lies in two distinct $\Omega_{t}, \Omega_{q}$. Then a similar argument, with $\Tensor{D} = \Norm{\Tensor{C}^{(t)}}{\Frob(\{\omega\})}$ and $\Tensor{E} = \Norm{\Tensor{C}^{(q)}}{\Frob(\{\omega\})}$, leads us to
    \begin{align*}
        \sum_{\Pind{i} \in \PindSet^{\Pind{m}}(\Omega)} \prod_{\tau \in [\kappa]} \Tensor{C}^{(\tau)}_{\Pind{i}_{\Omega_\tau}} &= \sum_{\Pind{j} \in \PindSet^{\Pind{m}}(\Omega \setminus \{ \omega \})} \bigg( \sum_{\Pind{\ell} \in \PindSet^{\Pind{m}}(\{ \omega \})} \Tensor{C}^{(t)}_{\Pind{\ell} \PindJoin \Pind{j}_{\Omega_t}} \Tensor{C}^{(q)}_{\Pind{\ell} \PindJoin \Pind{j}_{\Omega_q}} \bigg) \prod_{\tau \neq t,q} \Tensor{C}^{(\tau)}_{\Pind{j}_{\Omega_\tau}} \\
        &\leq \sum_{\Pind{j} \in \PindSet^{\Pind{m}}(\Omega \setminus \{ \omega \})} \big( \Tensor{D}_{\Pind{j}_{\Omega_t}} \Tensor{E}_{\Pind{j}_{\Omega_q}} \big) \prod_{\tau \neq t} \Tensor{C}^{(\tau)}_{\Pind{j}_{\Omega_\tau}}.
    \end{align*}
    We can build a partition $\pi'$ of $\Omega' \oplus \Omega'$ with the number of cells equal to
    \begin{equation*}
        |\pi'| = \begin{cases}
            \kappa - 2 & \text{if both}~\mho_t, \mho_q~\text{are singletons}, \\
            \kappa - 1 & \text{if one of}~\mho_t, \mho_q~\text{is a singleton}, \\
            \kappa & \text{if neither of}~\mho_t, \mho_q~\text{is a singleton}.
        \end{cases}
    \end{equation*}
    It remains to use the induction hypothesis and \cref{lemma:associative_trace_frob}.
\end{proof}

\begin{lemma}
\label{lemma:number_of_pairs}
Let $\varnothing \neq \Omega \subseteq [N]$, $\pi \in \Part_\kappa (\Omega \oplus \Omega)$, and $\Theta_\pi = \sum_{\omega \in \Omega} \theta_{\pi, \omega}$. Then
    \begin{equation*}
        2\Theta_\pi \leq \begin{cases}
            |\Omega|, & 1 \leq \kappa \leq |\Omega|, \\
            2|\Omega| - \kappa, & |\Omega| < \kappa \leq 2|\Omega|. 
        \end{cases}
    \end{equation*}
\end{lemma}
\begin{proof}
    Note that $2 \Theta_\pi \leq |\Omega|$ for every $1 \leq \kappa \leq 2|\Omega|$ since there are $|\Omega|$ pairs of $\omega$ and $\omega+N$ in total. Assume $\kappa > |\Omega|$ and denote $\Tilde{\pi} = \set{\mho \in \pi}{|\mho| = 1}$. Then
    \begin{equation*}
        |\Omega \oplus \Omega| = 2|\Omega| = \sum_{\mho \in \pi} |\mho| = |\Tilde{\pi}| + \sum_{\mho \in \pi \setminus \Tilde{\pi}} |\mho| \geq |\Tilde{\pi}| + 2(\kappa - |\Tilde{\pi}|)
    \end{equation*}
    and $|\Tilde{\pi}| \geq 2(\kappa - |\Omega|)$. It follows that there are at most $2|\Omega| - |\Tilde{\pi}| \leq 4|\Omega| - 2\kappa$ dimensions to be paired in the remaining $\kappa - |\Tilde{\pi}|$ cells, hence at most $2|\Omega| - \kappa$ pairs can be formed.
\end{proof}

\subsection{Analysis of the specific quadratic form}

Recall that $\Delta(\Matrix{R}_1, \ldots, \Matrix{R}_{d-1})$ of \cref{eq:quadratic_form} is parametrized by $d$ matrices $\Matrix{W}_s \in \Real^{k_{s-1} \times k_s}$ with $k_0 = k_d = 1$ and takes $\Matrix{R}_{s} \in \Real^{k_s \times r_s}$ as input. Let $\Pind{k} = (k_1, \ldots, k_{d-1})$ and $\Pind{r} = (r_1, \ldots, r_{d-1})$. Let us define
\begin{gather*}
    \Tensor{C} = \Phi_{[d-1]}(\Matrix{W}_1, \ldots, \Matrix{W}_d) \in \Real^{\Pind{k}^{\times 2}}, \\
    \Tensor{C}_{\Pind{i} \PindCat \Pind{j}} = \Matrix{W}_1(1,\Pind{i}_1) \Matrix{W}_2(\Pind{j}_1,\Pind{i}_2) \ldots \Matrix{W}_d(\Pind{j}_{d-1},1), \quad \Pind{i}, \Pind{j} \in \PindSet^{\Pind{k}}.
\end{gather*}
More generally, let $\Omega = \{ \omega_1, \ldots, \omega_{|\Omega|} \} \subseteq [d-1]$ with $\omega_1 < \cdots < \omega_{|\Omega|}$ and take $|\Omega| + 1$ matrices $\widehat{\Matrix{W}}_1 \in \Real^{1 \times k_{\omega_1}}$, $\widehat{\Matrix{W}}_s \in \Real^{k_{\omega_{s-1}} \times k_{\omega_s}}$, $\widehat{\Matrix{W}}_{|\Omega| + 1} \in \Real^{k_{\omega_{|\Omega|}} \times 1}$. We define
\begin{gather*}
    \Tensor{\widehat{C}} = \Phi_{\Omega}(\widehat{\Matrix{W}}_1, \ldots, \widehat{\Matrix{W}}_{|\Omega| + 1}) \in \Real^{\Pind{k}^{\times 2}}(\Omega \oplus \Omega), \\
    \Tensor{\widehat{C}}_{\Pind{i} \PindCat \Pind{j}} = \widehat{\Matrix{W}}_1(1,\Pind{i}_{\omega_1}) \widehat{\Matrix{W}}_2(\Pind{j}_{\omega_1},\Pind{i}_{\omega_2}) \ldots \widehat{\Matrix{W}}_{|\Omega|+1}(\Pind{j}_{\omega_{|\Omega|}},1), \quad \Pind{i}, \Pind{j} \in \PindSet^{\Pind{k}}(\Omega), \\
    \Tensor{\widehat{D}} = \Psi_{\Omega}(\widehat{\Matrix{W}}_1, \ldots, \widehat{\Matrix{W}}_{|\Omega| + 1}) \in \Real^{\PindKron{\Pind{r}}{\Pind{k}}^{\times 2}}(\Omega \oplus \Omega), \\
    \Tensor{\widehat{D}}_{\PindKron{\Pind{\alpha}}{\Pind{i}} \PindCat \PindKron{\Pind{\beta}}{\Pind{j}}} = \Tensor{\widehat{C}}_{\Pind{i} \PindCat \Pind{j}} \prod\nolimits_{\omega \in \Omega} \delta_{\Pind{\alpha}_\omega, \Pind{\beta}_\omega}, \quad \Pind{i}, \Pind{j} \in \PindSet^{\Pind{k}}(\Omega), \quad \Pind{\alpha}, \Pind{\beta} \in \PindSet^{\Pind{r}}(\Omega).
\end{gather*}

\begin{lemma}
\label{lemma:specific_array}
Let $\Tensor{D} = \Psi_{[d-1]}(\Matrix{W}_1, \ldots, \Matrix{W}_d)$. Then
\begin{equation*}
    \Delta(\Matrix{R}_1, \ldots, \Matrix{R}_{d-1}) = \sum_{\Pind{i}, \Pind{j} \in \PindSet^{\Pind{k}}} \sum_{\Pind{\alpha}, \Pind{\beta} \in \PindSet^{\Pind{r}}} \Tensor{D}_{\PindKron{\Pind{\alpha}}{\Pind{i}} \PindCat \PindKron{\Pind{\beta}}{\Pind{j}}} \prod_{s = 1}^{d-1} \Matrix{R}_s(\Pind{i}_s, \Pind{\alpha}_s) \Matrix{R}_s(\Pind{j}_s, \Pind{\beta}_s).
\end{equation*}
\end{lemma}
\begin{proof}
Let $\Tensor{C} = \Phi_{[d-1]}(\Matrix{W}_1, \ldots, \Matrix{W}_d)$. Then the right-hand side is equal to
\begin{align*}
    \sum_{\Pind{i}, \Pind{j} \in \PindSet^{\Pind{k}}} &\sum_{\Pind{\alpha} \in \PindSet^{\Pind{r}}} \Tensor{C}_{\Pind{i} \PindCat \Pind{j}} \prod_{s = 1}^{d-1} \Matrix{R}_s(\Pind{i}_s, \Pind{\alpha}_s) \Matrix{R}_s(\Pind{j}_s, \Pind{\alpha}_s) \\
    = &\sum_{\Pind{\alpha} \in \PindSet^{\Pind{r}}} [\Matrix{W}_1 \Matrix{R}_1](1, \Pind{\alpha}_1) \cdot [\Matrix{R}_1^\trans \Matrix{W}_2 \Matrix{R}_2](\Pind{\alpha}_1, \Pind{\alpha}_2) \ldots [\Matrix{R}_{d-1}^\trans \Matrix{W}_d](\Pind{\alpha}_{d-1}, 1).
\end{align*}
\end{proof}

Our first step is to understand what the partial traces of $\Psi_{[d-1]}(\Matrix{W}_1, \ldots, \Matrix{W}_d)$ are. Thanks to \cref{lemma:associative_trace_frob}, it suffices to study partial traces along singletons. Recall that $\Compl{\Omega} = [d-1] \setminus \Omega$ denotes the complement of $\Omega \subseteq [d-1]$.

\begin{lemma}
\label{lemma:specific_trace_singleton}
Let $\Tensor{C} = \Phi_{[d-1]}(\Matrix{W}_1, \ldots, \Matrix{W}_d)$ and $\Tensor{D} = \Psi_{[d-1]}(\Matrix{W}_1, \ldots, \Matrix{W}_d)$. Then
\begin{equation*}
\begin{aligned}
    \Tensor{C}^{\langle\{s\}\rangle} &= \Phi_{\Compl{\{s\}}}(\Matrix{W}_1, \ldots, \Matrix{W}_{s} \Matrix{W}_{s+1}, \ldots, \Matrix{W}_d), \\
    \Tensor{D}^{\langle\{s\}\rangle} &= \Psi_{\Compl{\{s\}}}(\Matrix{W}_1, \ldots, \Matrix{W}_{s} \Matrix{W}_{s+1}, \ldots, \Matrix{W}_d) \cdot r_s,
\end{aligned}
\quad s \in [d-1].
\end{equation*}
\end{lemma}
\begin{proof}
Let $\Pind{i}, \Pind{j} \in \PindSet^{\Pind{k}}(\Compl{\{s\}})$. By the definition of partial traces,
\begin{align*}
    \Tensor{C}^{\langle\{s\}\rangle}_{\Pind{i} \PindCat \Pind{j}} &= \sum_{\Pind{\ell} \in \PindSet^{\Pind{k}}(\{s\})} \Matrix{W}_1(1,\Pind{i}_1) \ldots \Matrix{W}_{s}(\Pind{j}_{s-1},\Pind{\ell}_{s}) \Matrix{W}_{s+1}(\Pind{\ell}_{s},\Pind{i}_{s+1}) \ldots \Matrix{W}_d(\Pind{j}_{d-1},1) \\
    &= \Matrix{W}_1(1,\Pind{i}_1) \ldots [\Matrix{W}_{s} \Matrix{W}_{s+1}](\Pind{j}_{s-1},\Pind{i}_{s+1}) \ldots \Matrix{W}_d(\Pind{j}_{d-1},1),
\end{align*}
which proves the first part. Let $\Pind{\alpha}, \Pind{\beta} \in \PindSet^{\Pind{r}}(\Compl{\{s\}})$. Then, using \cref{lemma:kron_join}, we get
\begin{align*}
    \Tensor{D}^{\langle\{s\}\rangle}_{\PindKron{\Pind{\alpha}}{\Pind{i}} \PindCat \PindKron{\Pind{\beta}}{\Pind{j}}} &= \sum_{\PindKron{\Pind{\rho}}{\Pind{\ell}} \in \PindSet^{\PindKron{\Pind{r}}{\Pind{k}}}(\{s\})} \Tensor{D}_{(\PindKron{\Pind{\alpha}}{\Pind{i}} \PindJoin \PindKron{\Pind{\rho}}{\Pind{\ell}}) \PindCat (\PindKron{\Pind{\beta}}{\Pind{j}} \PindJoin \PindKron{\Pind{\rho}}{\Pind{\ell}})} \\
    &= \sum_{\Pind{\rho} \in \PindSet^{\Pind{r}}(\{s\})} \sum_{\Pind{\ell} \in \PindSet^{\Pind{k}}(\{s\})} \Tensor{D}_{\PindKron{(\Pind{\alpha} \PindJoin \Pind{\rho})}{(\Pind{i} \PindJoin \Pind{\ell})} \PindCat \PindKron{(\Pind{\beta} \PindJoin \Pind{\rho})}{(\Pind{j} \PindJoin \Pind{\ell})}} \\
    &= \sum_{\Pind{\rho} \in \PindSet^{\Pind{r}}(\{s\})} \sum_{\Pind{\ell} \in \PindSet^{\Pind{k}}(\{s\})} \Tensor{C}_{(\Pind{i} \PindJoin \Pind{\ell}) \PindCat (\Pind{j} \PindJoin \Pind{\ell})} \prod_{t = 1}^{d-1} \delta_{(\Pind{\alpha} \PindJoin \Pind{\rho})_{t}, (\Pind{\beta} \PindJoin \Pind{\rho})_{t}} \\
    &= \sum_{\Pind{\rho} \in \PindSet^{\Pind{r}}(\{s\})} \Tensor{C}^{\langle\{s\}\rangle}_{\Pind{i} \PindCat \Pind{j}} \prod_{t \neq s} \delta_{\Pind{\alpha}_{t}, \Pind{\beta}_{t}} = r_s \cdot \Tensor{C}^{\langle\{s\}\rangle}_{\Pind{i} \PindCat \Pind{j}} \prod_{t \neq s} \delta_{\Pind{\alpha}_{t}, \Pind{\beta}_{t}}.
\end{align*}
\end{proof}

As a consequence of \cref{lemma:associative_trace_frob,lemma:specific_trace_singleton}, we deduce that partial traces preserve the inherent structure of the arrays $\Phi_{[d-1]}(\Matrix{W}_1, \ldots, \Matrix{W}_d)$ and $\Psi_{[d-1]}(\Matrix{W}_1, \ldots, \Matrix{W}_d)$.

\begin{corollary}
\label{corollary:specific_trace}
For every $\varnothing \neq \Omega \subseteq [d-1]$,
\begin{align*}
    \Tensor{C}^{\langle\Compl{\Omega}\rangle} = \Phi_{\Omega}(\widehat{\Matrix{W}}_1, \ldots, \widehat{\Matrix{W}}_{|\Omega| + 1}), \quad
    \Tensor{D}^{\langle\Compl{\Omega}\rangle} = \Psi_{\Omega}(\widehat{\Matrix{W}}_1, \ldots, \widehat{\Matrix{W}}_{|\Omega| + 1}) \cdot \prod\nolimits_{\omega \in \Compl{\Omega}} r_\omega,
\end{align*}
where $\widehat{\Matrix{W}}_t = \prod_{s = s_{t}}^{s_{t+1}-1} \Matrix{W}_s$ with $1 = s_1 < s_2 < \cdots < s_{|\Omega|+2} = d + 1$.
\end{corollary}
\begin{proof}
    Apply \cref{lemma:specific_trace_singleton} successively for each $\omega \in \Compl{\Omega}$. This partitions the set $[d]$ into $|\Omega|+1$ intervals $[s_t, s_{t+1}-1]$. At each step, the interval to which $\omega$ belongs is joined with the neighboring interval on its right.
\end{proof}

This leads us to the next step: estimating the $\Norm{\cdot}{\pi}$ norms of such partial arrays.

\begin{lemma}
\label{lemma:specific_partition_norm}
Let $\Tensor{C} = \Phi_{[d-1]}(\Matrix{W}_1, \ldots, \Matrix{W}_d)$ and $\Tensor{D} = \Psi_{[d-1]}(\Matrix{W}_1, \ldots, \Matrix{W}_d)$. For every $\varnothing \neq \Omega \subseteq [d-1]$ and $\pi \in \Part_{\kappa}(\Omega \oplus \Omega)$,
\begin{align*}
    \Norm{\Tensor{C}^{\langle\Compl{\Omega}\rangle}}{\pi} &\leq \prod\nolimits_{s = 1}^{d} \Norm{\Matrix{W}_s}{\Frob}, \quad
    \Norm{\Tensor{D}^{\langle\Compl{\Omega}\rangle}}{\pi} \leq \Norm{\Tensor{C}^{\langle\Compl{\Omega}\rangle}}{\pi} \left(\prod\nolimits_{\omega \in \Compl{\Omega}} r_\omega\right) \left(\prod\nolimits_{\omega \in \Omega} r_\omega^{\theta_{\pi,\omega}}\right).
\end{align*}
\end{lemma}
\begin{proof}
For the first bound, we use \cref{eq:norm_hierarchy}, \cref{corollary:specific_trace}, and the submultiplicativity of the Frobenius norm for matrices:
\begin{equation*}
    \Norm{\Tensor{C}^{\langle\Compl{\Omega}\rangle}}{\pi} \leq \Norm{\Tensor{C}^{\langle\Compl{\Omega}\rangle}}{\Frob} = \prod\nolimits_{t = 1}^{|\Omega|+1} \Norm{\widehat{\Matrix{W}}_{t}}{\Frob} \leq \prod\nolimits_{s = 1}^{d} \Norm{\Matrix{W}_s}{\Frob}.
\end{equation*}
Next, let $\pi = \{ \mho_{\tau} \}_{\tau = 1}^{\kappa}$. By definition, the norm $\Norm{\Tensor{D}^{\langle\Compl{\Omega}\rangle}}{\pi}$ is equal to
\begin{equation*}
    \sup\set[\Bigg]{\sum_{\Pind{i}, \Pind{j} \in \PindSet^{\Pind{k}}(\Omega)} \sum_{\Pind{\alpha}, \Pind{\beta} \in \PindSet^{\Pind{r}}(\Omega)} \Tensor{D}^{\langle\Compl{\Omega}\rangle}_{\PindKron{\Pind{\alpha}}{\Pind{i}} \PindCat \PindKron{\Pind{\beta}}{\Pind{j}}} \prod_{\tau = 1}^{\kappa} \Tensor{Z}^{(\tau)}_{(\PindKron{\Pind{\alpha}}{\Pind{i}} \PindCat \PindKron{\Pind{\beta}}{\Pind{j}})_{\mho_\tau}}}{\Tensor{Z}^{(\tau)} \in \Real^{\PindKron{\Pind{r}}{\Pind{k}}}(\mho_\tau),~\Norm{\Tensor{Z}^{(\tau)}}{\Frob} = 1}.
\end{equation*}
According to \cref{corollary:specific_trace} and the definition of $\Norm{\Tensor{C}^{\langle\Compl{\Omega}\rangle}}{\pi}$,
\begin{align*}
    \Norm{\Tensor{D}^{\langle\Compl{\Omega}\rangle}}{\pi} &= \prod_{\omega \in \Compl{\Omega}} r_\omega \cdot \sup_{\{ \Tensor{Z}^{(\tau)} \}} \left[ \sum_{\Pind{\alpha} \in \PindSet^{\Pind{r}}(\Omega)} \sum_{\Pind{i}, \Pind{j} \in \PindSet^{\Pind{k}}(\Omega)} \Tensor{C}^{\langle\Compl{\Omega}\rangle}_{\Pind{i} \PindCat \Pind{j}} \prod_{\tau = 1}^{\kappa} \Tensor{Z}^{(\tau)}_{(\PindKron{\Pind{\alpha}}{\Pind{i}} \PindCat \PindKron{\Pind{\alpha}}{\Pind{j}})_{\mho_\tau}} \right] \\
    &= \prod_{\omega \in \Compl{\Omega}} r_\omega \cdot \sup_{\{ \Tensor{Z}^{(\tau)} \}} \left[ \sum_{\Pind{\alpha} \in \PindSet^{\Pind{r}}(\Omega)} \sum_{\Pind{i}, \Pind{j} \in \PindSet^{\Pind{k}}(\Omega)} \Tensor{C}^{\langle\Compl{\Omega}\rangle}_{\Pind{i} \PindCat \Pind{j}} \prod_{\tau = 1}^{\kappa} [\Tensor{Z}^{(\tau)}]^{\{\Pind{\alpha} \PindCat \Pind{\alpha}\}_{\mho_\tau}}_{(\Pind{i} \PindCat \Pind{j})_{\mho_\tau}} \right] \\
    &\leq \prod_{\omega \in \Compl{\Omega}} r_\omega \cdot \sup_{\{ \Tensor{Z}^{(\tau)} \}} \left[ \sum_{\Pind{\alpha} \in \PindSet^{\Pind{r}}(\Omega)} \Norm{\Tensor{C}^{\langle\Compl{\Omega}\rangle}}{\pi} \prod_{\tau = 1}^{\kappa} \Norm{[\Tensor{Z}^{(\tau)}]^{\{\Pind{\alpha} \PindCat \Pind{\alpha}\}_{\mho_\tau}}}{\Frob} \right].
\end{align*}
Let $\Omega_\tau = (\mho_\tau \cup (\mho_\tau - N)) \cap \Omega$ and define $\Tensor{E}^{(\tau)} \in \Real^{\Pind{r}}(\Omega_\tau)$ by $\Tensor{E}^{(\tau)}_{\Pind{\alpha}} = \Norm{[\Tensor{Z}^{(\tau)}]^{\{\Pind{\alpha} \PindCat \Pind{\alpha}\}_{\mho_\tau}}}{\Frob}$ for $\alpha \in \PindSet^{\Pind{r}}(\Omega_\tau)$. Then we can apply \cref{lemma:cauchy_schwarz} to these auxiliary partial arrays:
\begin{align*}
    \Norm{\Tensor{D}^{\langle\Compl{\Omega}\rangle}}{\pi} &\leq \Norm{\Tensor{C}^{\langle\Compl{\Omega}\rangle}}{\pi} \prod_{\omega \in \Compl{\Omega}} r_\omega \cdot \sup_{\{ \Tensor{Z}^{(\tau)} \}} \left[ \sum_{\Pind{\alpha} \in \PindSet^{\Pind{r}}(\Omega)} \prod_{\tau = 1}^{\kappa} \Tensor{E}^{(\tau)}_{\Pind{\alpha}_{\Omega_\tau}} \right] \\
    &\leq \Norm{\Tensor{C}^{\langle\Compl{\Omega}\rangle}}{\pi} \prod_{\omega \in \Compl{\Omega}} r_\omega \cdot \sup_{\{ \Tensor{Z}^{(\tau)} \}} \left[ \prod_{\omega \in \Omega} r_\omega^{\theta_{\pi, \omega}} \prod_{\tau = 1}^{\kappa} \Norm{\Tensor{E}^{(\tau)}}{\Frob} \right].
\end{align*}
Finally, note that $\Norm{\Tensor{E}^{(\tau)}}{\Frob} \leq \Norm{\Tensor{Z}^{(\tau)}}{\Frob} = 1$.
\end{proof}

\begin{corollary}
\label{corollary:specific_partition_norm}
Let $\Tensor{D} = \Psi_{[d-1]}(\Matrix{W}_1, \ldots, \Matrix{W}_d)$ and $r = \max_{s \in [d-1]} r_s$. For every $\varnothing \neq \Omega \subseteq [d-1]$ and $\pi \in \Part_{\kappa}(\Omega \oplus \Omega)$,
\begin{equation*}
    \Norm{\Tensor{D}^{\langle\Compl{\Omega}\rangle}}{\pi} \leq r^{|\Compl{\Omega}| + \Theta_\pi} \prod\nolimits_{s = 1}^{d} \Norm{\Matrix{W}_s}{\Frob}.
\end{equation*}
\end{corollary}
\begin{proof}
Use \cref{lemma:specific_partition_norm} and the definition of $\Theta_\pi$ in \cref{lemma:number_of_pairs}.
\end{proof}

\subsection{Moment and tail bounds}

\begin{theorem}
\label{theorem:specific_moments_bound}
    Let $d \geq 2$. Let $\{ \Matrix{R}_s \}_{s = 1}^{d-1}$ be random matrices of size $\Matrix{R}_s \in \Real^{k_s \times r_s}$ with independent sub-Gaussian entries satisfying
    \begin{equation*}
        \Expectation \Matrix{R}_s(\Pind{i}_s, \Pind{\alpha}_s) = 0, \quad \Expectation |\Matrix{R}_s(\Pind{i}_s, \Pind{\alpha}_s)|^2 = \eta_s^2, \quad \Norm{\Matrix{R}_s(\Pind{i}_s, \Pind{\alpha}_s)}{\psi_2} \leq \eta_s L
    \end{equation*}
    for some constants $\eta_s > 0$ and $L \geq 1$ and all $\Pind{i} \in \PindSet^{\Pind{k}}$, $\Pind{\alpha} \in \PindSet^{\Pind{r}}$. Then there exists an absolute constant $\hat{c}_{d} > 0$ that depends only on $d$ such that the quadratic form $\Delta(\Matrix{R}_1, \ldots, \Matrix{R}_{d-1})$ \cref{eq:quadratic_form} satisfies, for all $p \geq 1$ and $r = \max_{s \in [d-1]} r_s$, the bounds
    \begin{equation*}
        \Norm{\Delta - \Expectation \Delta}{\Ell_p} \leq \hat{c}_{d} \prod_{s = 1}^{d-1} (\eta_s L \sqrt{r})^2 \prod_{t = 1}^{d} \Norm{\Matrix{W}_t}{\Frob} \sum_{\kappa = 1}^{2d-2} \frac{p^{\frac{\kappa}{2}}}{r^{\frac{\kappa}{2}}}.
    \end{equation*}
\end{theorem}
\begin{proof}
By \cref{lemma:specific_array}, the array of $\Delta$ is $\Tensor{D} = \Psi_{[d-1]}(\Matrix{W}_1, \ldots, \Matrix{W}_d)$. Combining \cref{theorem:moments_bound,corollary:specific_partition_norm}, we obtain, with $\Lambda = \prod_{s = 1}^{d-1} (\eta_s L \sqrt{r})^2 \prod_{t = 1}^{d} \Norm{\Matrix{W}_t}{\Frob}$,
\begin{equation*}
    \Norm{\Delta - \Expectation \Delta}{\Ell_p} \leq \mathring{c}_{d-1} \Lambda \sum_{\varnothing \neq \Omega \subseteq [d-1]} \sum_{\kappa = 1}^{2|\Omega|} p^{\frac{\kappa}{2}} \sum_{\pi \in \Part_{\kappa} (\Omega \oplus \Omega)} r^{\Theta_\pi - |\Omega|}.
\end{equation*}
\cref{lemma:number_of_pairs} provides an upper bound on $\Theta_\pi$ so that
\begin{align*}
    \Norm{\Delta - \Expectation \Delta}{\Ell_p} &\leq \mathring{c}_{d-1} \Lambda \sum_{\Omega} \left( \sum_{\kappa = 1}^{|\Omega|} \left[ \sum_{\pi \in \Part_{\kappa}} \frac{p^{\frac{\kappa}{2}}}{r^{ \frac{|\Omega|}{2}}} \right] + \frac{p^{ \frac{|\Omega|}{2}}}{r^{ \frac{|\Omega|}{2}}} \sum_{\tau = 1}^{|\Omega|} \left[ \sum_{\pi \in \Part_{|\Omega| + \tau}} \frac{p^{\frac{\tau}{2}}}{r^{\frac{\tau}{2}}} \right] \right) \\
    &\leq \tilde{c}_{d} \Lambda \sum_{\Omega} \frac{1}{r^{\frac{|\Omega|}{2}}} \left( \sum_{\kappa = 1}^{|\Omega|} p^{\frac{\kappa}{2}} + p^{\frac{|\Omega|}{2}} \sum_{\tau = 1}^{|\Omega|} \frac{p^{\frac{\tau}{2}}}{r^{\frac{\tau}{2}}} \right)
\end{align*}
with $\tilde{c}_{d} = \mathring{c}_{d-1} \cdot \max_{\kappa \in [2d-2]} \genfrac{\{}{\}}{0pt}{}{2d-2}{\kappa}$. Here, $\genfrac{\{}{\}}{0pt}{}{2d-2}{\kappa}$ are Stirling's numbers of the second kind and $|\Part_\kappa(\Omega \oplus \Omega)| = \genfrac{\{}{\}}{0pt}{}{2|\Omega|}{\kappa} \leq \genfrac{\{}{\}}{0pt}{}{2d-2}{\kappa}$. The expression under the outer sum depends only on $|\Omega|$, and we can rewrite it
in terms of the binomial coefficients:
\begin{align*}
    \Norm{\Delta - \Expectation \Delta}{\Ell_p} &\leq \check{c}_{d} \Lambda \sum_{m = 1}^{d-1} \frac{1}{r^{\frac{m}{2}}} \left( \sum_{\kappa = 1}^{m} p^{\frac{\kappa}{2}} + p^{\frac{m}{2}} \sum_{\tau = 1}^{m} \frac{p^{\frac{\tau}{2}}}{r^{\frac{\tau}{2}}} \right)
\end{align*}
with $\check{c}_{d} = \tilde{c}_{d} \cdot \max_{m \in [d-1]} \genfrac{(}{)}{0pt}{1}{d-1}{m}$. Rearranging the first sum, we note that
\begin{equation*}
    \sum_{m = 1}^{d-1} \frac{1}{r^{\frac{m}{2}}} \sum_{\kappa = 1}^{m} p^{\frac{\kappa}{2}} = \sum_{\kappa = 1}^{d-1} p^{\frac{\kappa}{2}} \sum_{m = \kappa}^{d-1} \frac{1}{r^{\frac{m}{2}}} \leq \sum_{\kappa = 1}^{d-1} (d-\kappa) \frac{p^{\frac{\kappa}{2}}}{r^{\frac{\kappa}{2}}}.
\end{equation*}
For the second sum, a counting argument gives
\begin{equation*}
    \sum_{m = 1}^{d-1} \frac{p^{\frac{m}{2}}}{r^{\frac{m}{2}}} \sum_{\tau = 1}^{m} \frac{p^{\frac{\tau}{2}}}{r^{\frac{\tau}{2}}} = \sum_{\tau = 2}^{d} \lfloor\tfrac{\tau}{2}\rfloor \frac{p^{\frac{\tau}{2}}}{r^{\frac{\tau}{2}}} + \sum_{\tau = d+1}^{2d-2} \lfloor\tfrac{2d - \tau}{2}\rfloor\frac{p^{\frac{\tau}{2}}}{r^{\frac{\tau}{2}}},
\end{equation*}
which leads to the following estimate:
\begin{equation*}
    \sum_{m = 1}^{d-1} \frac{1}{r^{\frac{m}{2}}}\left( \sum_{\kappa = 1}^{m} p^{\frac{\kappa}{2}} + p^{\frac{m}{2}} \sum_{\tau = 1}^{m} \frac{p^{\frac{\tau}{2}}}{r^{\frac{\tau}{2}}} \right) \leq \sum_{\kappa=1}^{2d-2} (d - \lceil\tfrac{\kappa}{2}\rceil) \frac{p^{\frac{\kappa}{2}}}{r^{\frac{\kappa}{2}}} \leq (d-1) \sum_{\kappa = 1}^{2d-2} \frac{p^{\frac{\kappa}{2}}}{r^{\frac{\kappa}{2}}}.
\end{equation*}
Set $\hat{c}_{d} = \check{c}_{d} \cdot (d-1)$ to finish the proof.
\end{proof}

\begin{corollary}
\label{corollary:specific_tail_bound}
Let $\prod_{s=1}^{d-1} \eta_s = r^{-\frac{d-1}{2}}$, and let $e$ be Euler's number. For $\varepsilon~\in~(0,1)$,
\begin{equation*}
    \Prob[\Big]{|\Delta - \Expectation \Delta| \geq \varepsilon \cdot e(2d-2) \hat{c}_d L^{2d-2} \prod\nolimits_{t = 1}^{d} \Norm{\Matrix{W}_t}{\Frob} } \leq e \cdot \exp\left(  -r \varepsilon^2 \right).
\end{equation*} 
\end{corollary}
\begin{proof}
We use \cite[Lemma~13]{bamberger2022hanson} to obtain a tail bound from the moment bounds of \cref{theorem:specific_moments_bound}. It states that $\Prob[\big]{|\Delta - \Expectation \Delta| \geq z }$ is upper bounded for every $z > 0$ by
\begin{multline*}
    e \cdot \exp\left( -\min_{\kappa \in [2d-2]} \left[\frac{z r^{\frac{\kappa}{2}}}{e(2d-2) \hat{c}_d \prod_{s=1}^{d-1}(\eta_s L \sqrt{r})^{2} \prod_{t = 1}^{d} \Norm{\Matrix{W}_t}{\Frob}}\right]^{\frac{2}{\kappa}} \right) \\
    = e \cdot \exp\left( -r \min_{\kappa \in [2d-2]} \left[\frac{z }{e(2d-2) \hat{C}_d L^{2d-2} \prod_{t = 1}^{d} \Norm{\Matrix{W}_t}{F}}\right]^{\frac{2}{\kappa}} \right).
\end{multline*}
Let $z = \varepsilon \cdot e(2d-2) \hat{c}_d L^{2d-2} \prod_{t = 1}^{d} \Norm{\Matrix{W}_t}{\Frob}$ with $\varepsilon \in (0,1)$; then $\Prob[\big]{|\Delta - \Expectation \Delta| \geq z } \leq e \cdot \exp( -r \min_{\kappa \in [2d-2]} \varepsilon^{2/\kappa} ) \leq e \cdot \exp\left(  -r \varepsilon^2 \right)$.
\end{proof}

\subsection{Proof of (\ref{eq:tt_bound_fixed_factorization})}
Recall that we had a tensor $\Tensor{A} \in \Real^{n_1 \times \cdots \times n_d}$ and its TT factorization $\Tensor{A} = \TT(\Tensor{G}_1, \ldots, \Tensor{G}_d)$ with $\Tensor{G}_s \in \Real^{k_{s-1} \times n_s \times k_s}$. We approximate it with a tensor $\Tensor{B} \in \Real^{n_1 \times \cdots \times n_d}$ such that $\Tensor{B} = \TT(\Tensor{H}_1, \ldots, \Tensor{H}_d)$ with $\Tensor{H}_1(1,i_1,:) = \Tensor{G}_1(1,i_1,:) \Matrix{R}_1$, $\Tensor{H}_s(:,i_s,:) = \Matrix{R}_{s-1}^\trans \Tensor{G}_s(:,i_s,:) \Matrix{R}_s$, and $\Tensor{H}_d(:,i_d,1) = \Matrix{R}_{d-1}^\trans \Tensor{G}_d(:,i_d,1)$. The matrices $\Matrix{R}_s \in \Real^{k_s \times r_s}$ and the values of $r_s \in \N$ are to be chosen.

Let $\xi$ be a sub-Gaussian random variable with zero mean and unit variance. Choose positive weights $\{ \eta_s \}_{s = 1}^{d-1}$ so that $\prod_{s=1}^{d-1} \eta_s = r^{-\frac{d-1}{2}}$, where $r = \max_{s \in [d-1]} r_s$, and populate each matrix $\Matrix{R}_s$ with independent copies of $\eta_s \xi$. We can apply \cref{corollary:specific_tail_bound} to every individual entry of $\Tensor{B}$ to conclude that, for $\epsilon \in (0,1)$,
\begin{equation*}
    \Prob[\Big]{|\Tensor{B}(i_1, \ldots, i_d) - \Expectation \Tensor{B}(i_1, \ldots, i_d)| \geq \epsilon \cdot \sqrt{c_d} \prod\nolimits_{s = 1}^{d} \Norm{\Tensor{G}_s(:, i_s, :)}{\Frob} } \leq e \cdot \exp\left( -r \epsilon^2 \right)
\end{equation*}
with $\sqrt{c_d} = e(2d-2) \hat{c}_d \Norm{\xi}{\psi_2}^{2d-2}$. Taking a union bound over all entries, we get
\begin{equation*}
    \Prob[\Big]{\Norm{\Tensor{B} - \Expectation \Tensor{B}}{\max} \geq \epsilon \cdot \sqrt{c_d} \prod\nolimits_{s = 1}^{d} \Norm{\Tensor{G}_s}{\Frob,\infty} } \leq e \prod\nolimits_{s = 1}^{d} n_d \cdot \exp\left( -r \epsilon^2 \right).
\end{equation*}
Assume that $r \geq \log(2e \prod_{s = 1}^{d} n_s) / \epsilon^2$. Then the above probability is at most $1/2$, and there exist (full-rank) matrices $\Matrix{R}_s$ such that
\begin{equation*}
    \Norm{\Tensor{B} - \Expectation \Tensor{B}}{\max} \leq \epsilon \cdot \sqrt{c_d} \prod\nolimits_{s = 1}^{d} \Norm{\Tensor{G}_s}{\Frob,\infty}.
\end{equation*}
Let $\epsilon \leq 1 / \sqrt{c_d}$ and write it as $\epsilon = \varepsilon / \sqrt{c_d}$ for $\varepsilon \in (0,1)$. This brings the error and rank bounds into the sought form. Next, note that $\Expectation [\Matrix{R}_s \Matrix{R}_s^\trans] = \eta_s^2 r_s \Matrix{I}_{k_s}$ and
\begin{equation*}
    \Expectation \Tensor{B} = \prod\nolimits_{s = 1}^{d} \eta_s^2 r_s \cdot \Tensor{A} = \frac{1}{r^{d-1}} \prod\nolimits_{s = 1}^{d-1} r_s \cdot \Tensor{A}.
\end{equation*}
So we must choose $r_1 = \cdots = r_{d-1} = r$ to have $\Expectation \Tensor{B} = \Tensor{A}$. This proves \cref{eq:tt_bound_fixed_factorization} and, together with the argument in \Cref{sec:proof_overview}, finishes the proof of \cref{theorem:tt_approx}.
\section{Error bound via partially orthogonal tensor-train factorizations}
\label{sec:core_coherence}

\subsection{Minimal partially orthogonal factorizations}
Let $\Tensor{A} \in \Real^{n_1 \times \cdots \times n_d}$. For every $s \in [d-1]$, let $\Matrix{A}^{<s>} \in \Real^{(n_1 \ldots n_s) \times (n_{s+1} \ldots n_d)}$ denote its $s$th unfolding matrix. We shall call $\Matrix{A}^{<} = \Matrix{A}^{<d-1>}$ and $\Matrix{A}^{>} = \Matrix{A}^{<1>}$ the left and right unfoldings, respectively.

Consider a TT factorization $\Tensor{A} = \TT(\Tensor{G}_1, \ldots, \Tensor{G}_d)$ of TT rank $(k_1, \ldots, k_{d-1})$. The unfoldings of $\Tensor{A}$ can be factorized as $\Matrix{A}^{<s>} = \Matrix{A}_{\leq s} \Matrix{A}_{\geq s+1}$ with \cite{holtz2012manifolds, budzinskiy2023tensor}
\begin{equation}
\label{eq:interface_matrices}
\begin{aligned}
    \Matrix{A}_{\leq s} &= (\Matrix{I}_{n_s} \otimes \Matrix{A}_{\leq s-1}) \Matrix{G}_s^{<} \in \Real^{(n_1 \ldots n_s) \times k_s}, & \Matrix{A}_{\leq 0} &= 1, \\
    \Matrix{A}_{\geq s+1} &= \Matrix{G}_{s+1}^{>} (\Matrix{I}_{n_{s+1}} \otimes \Matrix{A}_{\geq s+2}) \in \Real^{k_s \times (n_{s+1} \ldots n_d)}, & \Matrix{A}_{\geq d+1} &= 1.
\end{aligned}
\end{equation}
It follows that $\ttrank{\Tensor{A}} = (\rank{\Matrix{A}^{<1>}}, \ldots, \rank{\Matrix{A}^{<d-1>}})$ \cite{oseledets2011tensor, holtz2012manifolds}.
A TT factorization is called minimal if the unfoldings $\Matrix{G}_s^{<}$ and $\Matrix{G}_s^{>}$ are full-rank for every TT core $\Tensor{G}_s$. This is equivalent to $(k_1, \ldots, k_{d-1}) = \ttrank{\Tensor{A}}$ \cite{holtz2012manifolds}.

A tensor is said to be left-orthogonal if its left unfolding has orthonormal columns and right-orthogonal if its right unfolding has orthonormal rows. A TT factorization is called $t$-orthogonal for $t \in [d]$ if the TT cores $\{ \Tensor{G}_s \}_{s = 1}^{t-1}$ are left-orthogonal and $\{ \Tensor{G}_s \}_{s = t+1}^{d}$ are right-orthogonal---by \cref{eq:interface_matrices}, in this case $\Matrix{A}_{\leq t - 1}$ has orthonormal columns and $\Matrix{A}_{\geq t + 1}$ has orthonormal rows. We will refer to $1$-orthogonal TT factorizations as right-orthogonal and $d$-orthogonal as left-orthogonal.

All minimal $t$-orthogonal TT factorizations of a tensor can be transformed into one another with orthogonal matrices.

\begin{lemma}
\label{lemma:equivalent_orthogonal_tt_factorizations}
Let $\Tensor{A} \in \Real^{n_1 \times \cdots \times n_d}$ have $\ttrank{\Tensor{A}} = (r_1, \ldots, r_{d-1})$ and consider its minimal TT factorizations: left-orthogonal $\{ \Tensor{U}_s \}_{s = 1}^{d}$, right-orthogonal $\{ \Tensor{V}_s \}_{s = 1}^{d}$, and $t$-orthogonal $\{ \Tensor{T}_s \}_{s = 1}^{d}$ for $t \in [d]$. There exist orthogonal matrices $\{ \Matrix{Q}_s \}_{s = 1}^{d-1}$ of size $\Matrix{Q}_s \in \Real^{r_s \times r_s}$ such that
\begin{align*}
    \Matrix{T}_s^{<} &= (\Matrix{I}_{n_s} \otimes \Matrix{Q}_{s-1}^\trans) \Matrix{U}_s^{<} \Matrix{Q}_s, \quad s = 1, \ldots, t-1, \\
    \Matrix{T}_s^{>} &= \Matrix{Q}_{s-1}^\trans \Matrix{V}_s^{>} (\Matrix{I}_{n_s} \otimes \Matrix{Q}_{s}),\quad s = t+1, \ldots, d,
\end{align*}
where $\Matrix{Q}_0 = \Matrix{Q}_d = 1$, and one of the following three conditions holds:
\begin{align*}
    \Matrix{T}_{t}^{<} &= (\Matrix{I}_{n_t} \otimes \Matrix{Q}_{t-1}^\trans) \Matrix{U}_t^{<}, & & & t = d, \\
    & & \Matrix{T}_{t}^{>} &= \Matrix{V}_t^{>} (\Matrix{I}_{n_t} \otimes \Matrix{Q}_{t}), & t = 1, \\
    \mathrm{col}(\Matrix{T}_t^{<}) &= \mathrm{col}\big((\Matrix{I}_{n_t} \otimes \Matrix{Q}_{t-1}^\trans) \Matrix{U}_t^{<}\big), & \mathrm{row}(\Matrix{T}_t^{>}) &= \mathrm{row} \big(\Matrix{V}_t^{>} (\Matrix{I}_{n_t} \otimes \Matrix{Q}_{t})\big), & 1 < t < d.
\end{align*}
\end{lemma}
\begin{proof}
If $t = 1$ or $t = d$, the statement is exactly \cite[Theorem~1(b)]{holtz2012manifolds}. Let $1 < t < d$ and rearrange $\Tensor{A}$ into an order-$(t+1)$ tensor $\Tensor{\widehat{A}}$ of size $n_1 \times \cdots \times n_t \times (n_{t+1} \ldots n_d)$. Note that $\ttrank{\Tensor{\widehat{A}}} = (r_1, \ldots, r_t)$ because the unfoldings of $\Tensor{\widehat{A}}$ coincide with those of $\Tensor{A}$. Consider tensors (essentially matrices) $\Tensor{G}_{t+1}, \Tensor{H}_{t+1} \in \Real^{r_t \times (n_{t+1} \ldots n_d) \times 1}$ obtained by recursive multiplication of TT cores $\{ \Tensor{U}_s \}_{s = t+1}^{d}$ and $\{ \Tensor{T}_s \}_{s = t+1}^{d}$, respectively, as in the second row of \cref{eq:interface_matrices}. Define a left-orthogonal $\Tensor{P}_t \in \Real^{r_{t-1} \times n_t \times r_t}$ by taking $\Matrix{P}_t^{<}$ from the QR decomposition $\Matrix{T}_t^{<} = \Matrix{P}_t^{<} \Matrix{R}$. Then $\Tensor{\widehat{A}} = \TT(\Tensor{U}_1, \ldots, \Tensor{U}_t, \Tensor{G}_{t+1}) = \TT(\Tensor{T}_1, \ldots, \Tensor{T}_{t-1}, \Tensor{P}_t, \Matrix{R}\Tensor{H}_{t+1})$ are minimal left-orthogonal TT factorizations, and \cite[Theorem~1(b)]{holtz2012manifolds} guarantees that the sought orthogonal matrices $\{ \Matrix{Q}_s \}_{s = 1}^{t-1}$ exist and $\Matrix{P}_t^{<} = (\Matrix{I}_{n_t} \otimes \Matrix{Q}_{t-1}^\trans) \Matrix{U}_t^{<} \Matrix{\tilde{Q}}_t$ for an orthogonal $\Matrix{\tilde{Q}}_t$. A similar argument for the right-orthogonal factorizations shows that $\{ \Matrix{Q}_s \}_{s = t}^{d-1}$ exist too.
\end{proof}

Minimal left- and right-orthogonal TT factorizations can be computed for every tensor with the TT-SVD algorithm \cite{oseledets2011tensor, holtz2012manifolds}; to get a minimal $t$-orthogonal factorization, TT-SVD can be followed with a partial sweep of QR (or LQ) orthogonalizations.

\subsection{Tensor-train core coherences}
Let $\mathrm{M} \subseteq \Real^{m}$ be a $q$-dimensional subspace with an orthonormal basis given by the columns of $\Matrix{Q} \in \Real^{m \times q}$. Its coherence is \cite{candes2009exact}
\begin{equation*}
    \mu(\mathrm{M}) = \tfrac{m}{q} \max_{i \in [m]} \Norm{\Matrix{Q} \Matrix{Q}^\trans \Matrix{e}_i }{2}^2 = \tfrac{m}{q} \Norm{\Matrix{Q}}{2,\infty}^2
\end{equation*}
and is independent of the choice of the orthonormal basis $\Matrix{Q}$. We can associate two values of coherence with every matrix $\Matrix{A}$, that of its column and row spaces:
\begin{equation*}
    \mu_{<}(\Matrix{A}) = \mu(\mathrm{col}(\Matrix{A})), \quad \mu_{>}(\Matrix{A}) = \mu(\mathrm{row}(\Matrix{A})).
\end{equation*}

Assume that $m$ is divisible by $p \in \N$, $n = \tfrac{m}{p}$, and let $\Norm{\cdot}{\mathrm{u}}$ be a unitarily invariant matrix norm \cite{horn1994topics}. We define the $p$-block coherence of $\mathrm{M}$ with respect to $\Norm{\cdot}{\mathrm{u}}$ as
\begin{equation*}
    \mu_{\mathrm{u}, p}(\mathrm{M}) = \tfrac{m}{q} \max_{i \in [n]} \Norm[\Big]{\Matrix{Q}^\trans \begin{bmatrix}
        \Matrix{e}_{1 + (i-1)p} & \Matrix{e}_{2 + (i-1)p} & \cdots & \Matrix{e}_{p + (i-1)p}
    \end{bmatrix} }{\mathrm{u}}^2.
\end{equation*}
This quantity is also invariant to the choice of $\Matrix{Q}$ since the norm is unitarily invariant. We use $p$-block coherence to define left and right coherences of a tensor $\Tensor{G} \in \Real^{p \times n \times q}$:
\begin{equation*}
    \mu_{\mathrm{u}, <}(\Tensor{G}) = \mu_{\mathrm{u},p}(\mathrm{col}(\Matrix{G}^{<})), \quad \mu_{\mathrm{u},>}(\Tensor{G}) = \mu_{\mathrm{u},q}(\mathrm{row}(\Matrix{G}^{>})).
\end{equation*}

\begin{lemma}
\label{lemma:lr_coherence_orthogonal_core}
Let $\Tensor{U}, \Tensor{V} \in \Real^{p \times n \times q}$ be left- and right-orthogonal, respectively. Then
\begin{equation*}
    \mu_{\mathrm{u},<}(\Tensor{U}) = \tfrac{np}{q} \Norm{\Tensor{U}}{\mathrm{u}, \infty}^2, \quad \mu_{\mathrm{u},>}(\Tensor{V}) = \tfrac{nq}{p} \Norm{\Tensor{V}}{\mathrm{u}, \infty}^2.
\end{equation*}
\end{lemma}
\begin{proof}
The columns of $\Matrix{U}^{<}$ are an orthonormal basis of $\mathrm{col}(\Matrix{U}^{<})$. Then
\begin{equation*}
    \mu_{\mathrm{u},<}(\Tensor{U}) = \tfrac{np}{q} \max_{i \in [n]} \Norm{\Tensor{U}(:,i,:)^\trans}{\mathrm{u}}^2 = \tfrac{np}{q} \max_{i \in [n]} \Norm{\Tensor{U}(:,i,:)}{\mathrm{u}}^2 = \tfrac{np}{q} \Norm{\Tensor{U}}{\mathrm{u}, \infty}^2.
\end{equation*}
The second equality holds since $\Norm{\cdot}{\mathrm{u}}$ is unitarily invariant. The proof for $\Tensor{V}$ is similar.
\end{proof}

\begin{lemma}
\label{lemma:tt_core_coherence_invariant}
Consider minimal TT factorizations of a tensor: left-ortho\-go\-nal $\{ \Tensor{U}_s \}_{s = 1}^{d}$, right-orthogonal $\{ \Tensor{V}_s \}_{s = 1}^{d}$, and $t$-orthogonal $\{ \Tensor{T}_s \}_{s = 1}^{d}$ for $t \in [d]$. Then
\begin{align*}
    \mu_{\mathrm{u},<}(\Tensor{T}_s) &= \mu_{\mathrm{u},<}(\Tensor{U}_s), \quad s = 1, \ldots, \min\{t, d-1\}, \\
    \mu_{\mathrm{u},>}(\Tensor{T}_s) &= \mu_{\mathrm{u},>}(\Tensor{V}_s), \quad s = \max\{2,t\}, \ldots, d.
\end{align*}
\end{lemma}
\begin{proof}
    When $1 < s < d$, or $s = 1$ and $t > 1$, or $s = d$ and $t < d$, use the definition of left and right coherences, \cref{lemma:equivalent_orthogonal_tt_factorizations,lemma:lr_coherence_orthogonal_core}, and that unitary invariance of $\Norm{\cdot}{\mathrm{u}}$. If $s = t = 1$, then it follows from \cref{eq:interface_matrices} that the first unfolding of the tensor can be factorized as $\Matrix{U}_1^{<} \Matrix{X} = \Matrix{T}_1^{<} \Matrix{Y}$ with full-rank $\Matrix{X}$ and $\Matrix{Y}$. Then $\mathrm{col}(\Matrix{U}_1^{<}) = \mathrm{col}(\Matrix{T}_1^{<})$ and $\mu_{\mathrm{u},<}(\Tensor{T}_1) = \mu_{\mathrm{u},<}(\Tensor{U}_1)$. A similar argument applies to the case $s = t = d$.
\end{proof}

\cref{lemma:tt_core_coherence_invariant} shows that for a minimal $t$-orthogonal TT factorization, the left coherence of each left-orthogonal TT core and the right coherence of each right-orthogonal TT core depends \textit{on the tensor itself} and not on the specific factorization.

This motivates the following definition. Let $\Tensor{A}$ be an order-$d$ tensor with a minimal left-orthogonal TT factorization $\{ \Tensor{U}_s \}_{s = 1}^{d}$ and a minimal right-orthogonal TT factorization $\{ \Tensor{V}_s \}_{s = 1}^{d}$. We define the $s$th \textit{left and right TT core coherences} of $\Tensor{A}$ with respect to a unitarily invariant norm $\Norm{\cdot}{\mathrm{u}}$ as
\begin{align*}
    \mu_{\mathrm{u},<}^{(s)}(\Tensor{A}) &= \mu_{\mathrm{u},<}(\Tensor{U}_s), \quad s = 1, \ldots, d-1, \\
    \mu_{\mathrm{u},>}^{(s)}(\Tensor{A}) &= \mu_{\mathrm{u},>}(\Tensor{V}_s), \quad s = 2, \ldots, d.   
\end{align*}

\begin{remark}
TT core coherences were introduced in \cite{budzinskiy2023tensor} for the spectral norm in the context of TT completion. We extend them to arbitrary unitarily invariant norms.
\end{remark}

\subsection{Entrywise tensor-train approximation}
We can obtain an upper bound on $\gamma_{\mathrm{u}}^{\TT}(\Tensor{A})$ from \cref{eq:gammatt} in terms of the TT core coherences of $\Tensor{A}$.

\begin{lemma}
\label{lemma:gammatt_bound_via_tt_coherence}
Let $\Tensor{A} = \TT(\Tensor{T}_1, \ldots, \Tensor{T}_d) \in \Real^{n_1 \times \cdots \times n_d}$ be a minimal $t$-orthogonal TT factorization for $t \in [d]$. Let $\ttrank{A} = (r_1, \ldots, r_{d-1})$ and $r_0 = r_d = 1$. Then
\begin{align*}
    \Norm{\Tensor{T}_s}{\mathrm{u},\infty} &= \sqrt{\tfrac{r_s}{n_s r_{s-1}} \mu_{\mathrm{u},<}^{(s)}(\Tensor{A})}, & s = 1, \ldots, t-1, \\
    \Norm{\Tensor{T}_s}{\mathrm{u},\infty} &= \sqrt{\tfrac{r_{s-1}}{n_s r_s} \mu_{\mathrm{u},>}^{(s)}(\Tensor{A})}, & s = t+1, \ldots, d, \\
    \Norm{\Tensor{T}_t}{\mathrm{u},\infty} &\leq \Norm{\Matrix{A}^{<t>}}{2} \sqrt{\tfrac{r_t}{n_t r_{t-1}} \mu_{\mathrm{u},<}^{(t)}(\Tensor{A})}, & t < d, \\
    \Norm{\Tensor{T}_t}{\mathrm{u},\infty} &\leq  \Norm{\Matrix{A}^{<t-1>}}{2} \sqrt{\tfrac{r_{t-1}}{n_t r_t} \mu_{\mathrm{u},>}^{(t)}(\Tensor{A})}, & t > 1.
\end{align*}
\end{lemma}
\begin{proof}
The case $s \neq t$ follows from \cref{lemma:lr_coherence_orthogonal_core,lemma:tt_core_coherence_invariant}. If $s = t < d$, then according to \cref{eq:interface_matrices} the unfolding $\Matrix{A}^{<t>}$ can be factorized as $\Matrix{A}^{<t>} = \Matrix{P} \Matrix{T}_t^{<} \Matrix{Q}$, where $\Matrix{P}$ has orthonormal columns and $\Matrix{Q}$ has orthonormal rows, thus $\Matrix{A}^{<t>}$ and $\Matrix{T}_t^{<}$ have the same nonzero singular values. Consider a QR decomposition $\Matrix{T}_t^{<} = \Matrix{Q} \Matrix{R}$ and a third-order tensor $\Tensor{Q}$ such that $\Matrix{Q}^{<} = \Matrix{Q}$. Then \cite[Corollary 3.5.10]{horn1994topics}
\begin{equation*}
    \Norm{\Tensor{T}_t}{\mathrm{u},\infty} \leq \Norm{\Tensor{Q}}{\mathrm{u},\infty} \Norm{\Matrix{R}}{2} = \Norm{\Tensor{Q}}{\mathrm{u},\infty} \Norm{\Matrix{A}^{<t>}}{2}.
\end{equation*}
Note that $\mu_{\mathrm{u}, <}(\Tensor{Q}) = \mu_{\mathrm{u}, <}(\Tensor{T}_t)$ and use \cref{lemma:lr_coherence_orthogonal_core,lemma:tt_core_coherence_invariant}. The same argument applies to the case $s = t > 1$.
\end{proof}

\begin{corollary}
\label{corollary:tt_approx_coherence}
Let $\varepsilon \in (0,1)$ and $n_1, \ldots, n_d \in \N$. Let $r \in \N$ be given by \cref{eq:tt_rank}. For every $\Tensor{A} \in \Real^{n_1 \times \cdots \times n_d}$ of $\ttrank{A} = (r_1, \ldots, r_{d-1})$, there exists $\Tensor{B} \in \Real^{n_1 \times \cdots \times n_d}$ of $\ttrank{\Tensor{B}} \preccurlyeq r$ such that, with $\mu_<^{(s)} = \mu_{\Frob,<}^{(s)}(\Tensor{A})$ and $\mu_>^{(s)} = \mu_{\Frob,>}^{(s)}(\Tensor{A})$, 
\begin{equation*}
    \Norm{\Tensor{A} - \Tensor{B}}{\max} \leq \frac{\varepsilon}{\sqrt{n_1 \ldots n_d}} \cdot \min_{t \in [d-1]} \left\{ r_t \left( \prod\nolimits_{s = 1}^{t} \mu_{<}^{(s)} \right)^{\frac{1}{2}} \left( \prod\nolimits_{s = t+1}^{d} \mu_{>}^{(s)}\right)^{\frac{1}{2}} \Norm{\Matrix{A}^{<t>}}{2}\right\}.
\end{equation*}
\end{corollary}
\begin{proof}
For every $t \in [d-1]$, we can construct a $t$-orthogonal TT factorization of $\Tensor{A}$, use it to estimate $\gamma_{\Frob}^{\TT}(\Tensor{A})$ via \cref{lemma:gammatt_bound_via_tt_coherence}, and apply \cref{theorem:tt_approx}.
\end{proof}

Let us compare the estimate of the approximation error in \cref{corollary:tt_approx_coherence} for $d > 2$ and $d = 2$. For each unfolding $\Matrix{A}^{<t>}$ there exists a matrix $\Matrix{B}_t$ of $\rank{\Matrix{B}_t} \leq \lceil c_2 \log(2e \cdot \prod_{s=1}^{d}n_s) / \varepsilon^2 \rceil$ such that $\Norm{\Matrix{A}^{<t>} - \Matrix{B}_t}{\max} \leq \frac{\varepsilon \cdot r_t}{\sqrt{n_1 \ldots n_d}} \sqrt{\mu_{<}(\Matrix{A}^{<t>}) \mu_{>}(\Matrix{A}^{<t>})} \Norm{\Matrix{A}^{<t>}}{2}$.
The low-rank factors of $\Matrix{B}_t$ are not assumed to possess any specific structure. If we request them to have low-rank factorizations themselves, the quality of approximation should deteriorate. It follows from \cref{eq:interface_matrices} that $\mu_{<}(\Matrix{A}^{<t>}) \leq \prod_{s = 1}^{t} \mu_{\Frob,<}^{(s)}(\Tensor{A})$ and $\mu_{>}(\Matrix{A}^{<t>}) \leq \prod_{s = t+1}^{d} \mu_{\Frob,>}^{(s)}(\Tensor{A})$. 
The gap in these inequalities and the $\frac{c_d}{c_2}$-times growth of the rank bound are the cost of imposing the TT structure on the factors of $\Matrix{B}_t$.

\begin{example}
The CP representation of identity tensors from \cref{ex:eye_tensors}, seen as a TT factorization, is minimal and $t$-orthogonal for all $t \in [d]$ at once, and \cref{corollary:tt_approx_coherence} guarantees the same bounds as \cref{corollary:cp_approx}.
\end{example}

\begin{remark}
An artificial example based on random matrices was presented in \cite{budzinskiy2024distance} for the matrix version of \cref{corollary:cp_approx}. Perhaps a similar example could be created for $d > 2$ once the properties of tensors (their TT core coherences, specifically) with random TT cores are better studied.
\end{remark}
\section{Numerical experiments}
\label{sec:numerical}

The numerical computation of optimal low-rank approximation in the maximum norm is substantially more challenging than in the Frobenius norm. The problem is known to be NP-hard even in the simplest case of rank-one matrix approximation \cite{gillis2019low}: the algorithms with provable convergence to the global minimizer combine alternating minimization with an exhaustive search over the sign patterns of the initial condition \cite{morozov2023optimal}. Instead of computing the optimal entrywise approximation error, we rely on a simple \textit{heuristic} algorithm to bound it from above.

\subsection{Algorithm} The method of alternating projections is designed to find a point in the intersection of two sets by computing successive Euclidean projections:
\begin{equation*}
    \Proj{\mathbb{Y}}{\Tensor{X}} = \set[\Big]{\Tensor{Y} \in \mathbb{Y}}{\Norm{\Tensor{Y} - \Tensor{X}}{\Frob} = \inf_{\Tensor{Y'} \in \mathbb{Y}} \Norm{\Tensor{Y'} - \Tensor{X}}{\Frob}}, \quad \mathbb{Y} \subset \Real^{n_1 \times \cdots \times n_d}.
\end{equation*}
Our two sets are low-rank tensors and tensors close to $\Tensor{A}$ in the maximum norm:
\begin{equation*}
    \mathbb{M}_{\bm{r}} = \set{\Tensor{X}}{\ttrank{\Tensor{X}} = \bm{r}}, \quad \mathbb{B}_\varepsilon(\Tensor{A}) = \set{\Tensor{Y}}{\Norm{\Tensor{A} - \Tensor{Y}}{\max} \leq \varepsilon}.
\end{equation*}
Starting from $\Tensor{X}_0 \in \mathbb{M}_{\bm{r}}$, the algorithm computes successive Euclidean projections
\begin{equation*}
    \Tensor{Y}_{t+1} \in \Proj{\mathbb{B}_\varepsilon(\Tensor{A})}{\Tensor{X}_t}, \quad \Tensor{X}_{t+1} \in \Proj{\mathbb{M}_{\bm{r}}}{\Tensor{Y}_{t+1}}, \quad t = 0, 1, 2,\ldots,
\end{equation*}
and locally converges to $\mathbb{M}_{\bm{r}} \cap \mathbb{B}_\varepsilon(\Tensor{A})$ if the intersection is nonempty. In \cite[section 7.5]{budzinskiy2023quasioptimal} and \cite{budzinskiy2024distance}, alternating projections were combined with a binary search over $\varepsilon$ to obtain an upper bound on the optimal entrywise approximation error.

The Euclidean projections minimize the Frobenius norm, making each step of the algorithm tractable. The Euclidean projection onto $\mathbb{B}_\varepsilon(\Tensor{A})$ can be computed as $\Proj{\mathbb{B}_\varepsilon(\Tensor{A})}{\Tensor{X}} = \{ \Tensor{X} + \Proj{\mathbb{B}_\varepsilon(0)}{\Tensor{X} - \Tensor{A}} \}$ and amounts to clipping the entries of $\Tensor{X} - \Tensor{A}$ whose absolute value exceeds $\varepsilon$. The Euclidean projection onto $\mathbb{M}_{\bm{r}}$ is given by the truncated SVD for matrices, but is not computable for tensors. However, the TT-SVD algorithm achieves quasioptimal approximation \cite{oseledets2011tensor}, which motivated the use of quasioptimal alternating projections in \cite{sultonov2023low, budzinskiy2023quasioptimal} with TT-SVD in place of $\Proj{\mathbb{M}_{\bm{r}}}{\cdot}$.

\subsection{Experimental settings}
\label{sec:experimental_settings}
We use this approach\footnote{The code is available at \url{https://github.com/sbudzinskiy/LRAP}.} to bound the optimal entrywise TT approximation error for identity tensors (\cref{ex:eye_tensors}) and random tensors with independent entries distributed uniformly in $(-1,1)$. We present the results for $n \times n$ matrices and $n \times n \times n$ tensors approximated with rank $r$ and TT rank $(r,r)$.

In each experiment, we set the order $d$, the size $n$, and the approximation rank $r$; generate $\Tensor{A}$ as an identity or random uniform tensor; generate the initial condition $\Tensor{X}_0 \in \mathbb{M}_{\bm{r}}$ with random standard Gaussian TT cores and multiply it by $r^{1-d}$; use alternating projections with binary search to obtain the absolute error estimate $\varepsilon$. We repeat every experiment 5 times and report the minimum value of $\varepsilon$ for identity tensors and the median for random uniform tensors (new $\Tensor{A}$ in each of the 5 experiments).

We do not consider tensors of order $d > 3$ because of the high computational costs. For the same reasons, we restrict the experiments to $n \leq 800$. At the same time, the theoretical bounds of \cref{theorem:matrix_approx,theorem:tt_approx} become meaningful only for larger $n$. Let $r(n,\varepsilon)$ be given by \cref{eq:matrix_rank}; then the inequality $r(n,\varepsilon) < n$, which guarantees that the approximant is not full-rank, requires $n > 18694$ for $\varepsilon = 0.1$ (see \cref{tab:rank_values}).

\begin{table}[tbhp]
\footnotesize
\caption{Rank bound \cref{eq:matrix_rank} of \cref{theorem:matrix_approx} for specific values of $n$ and $\varepsilon$.}
\label{tab:rank_values}
\begin{center}
\begin{tabular}{|c|c|c|c|c|} \hline
$\vphantom{n^{n^n}} n$ & $18694$ & $10^5$ & $10^7$ & $10^9$ \\ \hline
$r(n,0.1)$ & $18694$ & $21713$ & $30002$ & $38291$ \\ \hline
\end{tabular}
\end{center}
\end{table}

\noindent This means that our numerical results should be perceived as complementary to \cref{theorem:matrix_approx,theorem:tt_approx} rather than illustrating their theoretical bounds. Numerical experiments for $d = 2$ and $n$ up to $20000$ can be found in \cite{budzinskiy2024distance}.

\subsection{Numerical results} In \cref{fig:eye_vs_uniform_varr}, we study how the entrywise approximation error depends on the approximation rank $r$. The rank bounds of \cref{theorem:matrix_approx,theorem:tt_approx} suggest the decay rate $r^{-1/2}$, but then the error must drop to zero as $r \to n$. We observe the rate $\mathcal{O}((n-r)^{\alpha} r^{-\beta})$ with $\alpha,\beta > 0$ for identity tensors and random uniform matrices, but have no solid theory to explain this behavior. For random tensors, the rate is better described as $\mathcal{O}(\mathrm{polylog}(n-r))$, and we suppose that $\mathcal{O}((n-r)^{\alpha} r^{-\beta})$ could be detected with larger $n$ and $r$.

Note how close the error curves are for identity matrices and identity tensors---this is a manifestation of the fact that the TT factorization quasinorm \eqref{eq:gammatt} of identity tensors is bounded by one for all $n$ and $d$. This is not the case for random tensors: the increase of the order of the tensor significantly impacts the approximation error.

\begin{figure}[tbhp]
\centering
	\includegraphics[width=0.99\textwidth]{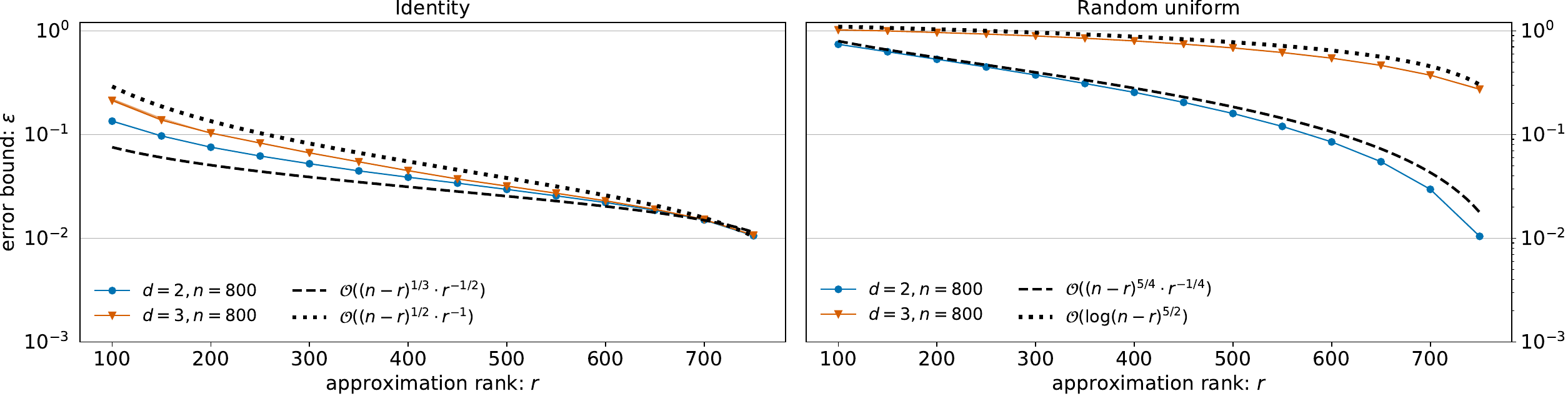}
\caption{Entrywise TT approximation errors computed with a combination of quasioptimal alternating projections and binary search for the varying approximation rank $r$.}
\label{fig:eye_vs_uniform_varr}
\end{figure}

In \cref{fig:eye_vs_uniform_varn}, we look at the growth of the entrywise approximation error as the size $n$ increases. For large $n$, \cref{theorem:matrix_approx,theorem:tt_approx} guarantee an error bound $\varepsilon \gamma_{\Frob}^{\TT}(\Tensor{A}_n)$ and rank bounds that grow as $\mathcal{O}(\log(n) / \varepsilon^2)$. With fixed (and large) approximation rank $r$, the guaranteed error is therefore $\mathcal{O}(\gamma_{\Frob}^{\TT}(\Tensor{A}_n) \sqrt{\log(n) / r})$. Meanwhile, the optimal entrywise approximation error is always upper bounded by $\Norm{\Tensor{A}_n}{\max}$, which can be attained with a vanishing sequence of low-rank tensors. Thus, the above error bound---which itself holds for large $n$ and $r$---necessarily loses meaning for extremely large $n$ if $\Norm{\Tensor{A}_n}{\max}$ are uniformly bounded. For identity tensors and relatively small $n$, we see in \cref{fig:eye_vs_uniform_varn} that the entrywise errors grow as $\mathcal{O}((n-r)^{\alpha})$ with the same values of $\alpha$ as in \cref{fig:eye_vs_uniform_varr}. Based on the numerical results of \cite{budzinskiy2024distance}, we expect the growth rate to slow down to $\mathcal{O}(\mathrm{polylog}(n-r))$ for larger $n$. We observe the rate $\mathcal{O}((n-r)^{\alpha})$ for random matrices too; the values of $\alpha$ are bigger than the corresponding values for identity matrices because the factorization norm of random matrices is not uniformly upper bounded for all $n$. The errors grow slower, as $\mathcal{O}(\mathrm{polylog}(n-r))$, for random tensors since they are close to the ultimate upper bound $\Norm{\Tensor{A}_n}{\max} = 1$.

\begin{figure}[tbhp]
\centering
	\includegraphics[width=0.99\textwidth]{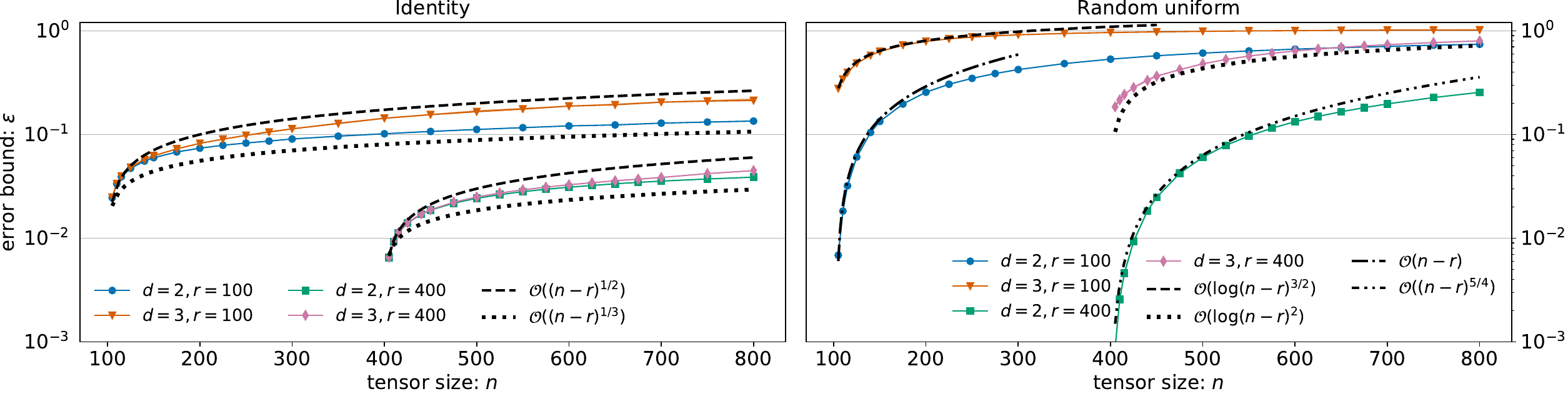}
\caption{Entrywise TT approximation errors computed with a combination of quasioptimal alternating projections and binary search for the varying tensor size $n$.}
\label{fig:eye_vs_uniform_varn}
\end{figure}
\section{Discussion}
\label{sec:dicussion}

\cref{theorem:tt_approx} provides a new a priori bound on the entrywise approximation error achievable with tensor networks. While we focused on the TT format---a tensor network with the simplest linear structure---our argument can be applied to more general tree tensor networks \cite{hackbusch2009new, grasedyck2010hierarchical, ye2018tensor, kramer2020tree}. Several modifications are, of course, in order. First, the analogue of $\gamma_{\Frob}^{\TT}$ \cref{eq:gammatt} needs to take into account the cores that depend on the ``inner'' indices, but not on the ``outer'' indices. For example, the appropriate formula for the Tucker factorization \cite{kolda2009tensor}
\begin{equation*}
    \Tensor{A}(i_1, \ldots, i_d) = \sum\nolimits_{\alpha_1, \ldots, \alpha_d} \Tensor{G}(\alpha_1, \ldots, \alpha_d) \Matrix{C}_1(i_1, \alpha_1) \ldots \Matrix{C}_d(i_d, \alpha_d)
\end{equation*}
would look like $\gamma^{\mathrm{Tucker}}(\Tensor{A}) = \inf\{ \Norm{\Tensor{G}}{\mathrm{sp}} \prod_{s = 1}^{d} \Norm{\Matrix{C}_s}{2,\infty}\}$ and involve the spectral~\mbox{tensor} norm \cite{wang2017operator}. Second, the analogues of $\Phi_{\Omega}$, $\Psi_{\Omega}$, and \cref{lemma:specific_trace_singleton} would reflect the graph structure of the tensor network with partial traces corresponding to edge contractions.

\cref{theorem:tt_approx} is most suitable for the analysis of families of tensors whose TT factorization quasinorm $\gamma_{\Frob}^{\TT}$ can be uniformly bounded with respect to the size of the tensor. The family of identity tensors in \cref{ex:eye_tensors} and function-generated tensors of \cite{budzinskiy2024big} are illustrative examples. The TT factorization quasinorm could be of interest on its own with potential applications in tensor completion (cf. \cite{ghadermarzy2019near, harris2021deterministic,cao20241}).

\section*{Acknowledgements}
I thank Stefan Bamberger for the discussion of \cite[Theorem 3]{bamberger2022hanson} and its proof; Vladimir Kazeev for reading parts of the text and suggesting improvements; Dmitry Kharitonov for correcting a sum in the proof of  \cref{theorem:specific_moments_bound};~and the anonymous referees for their comments that helped me improve the presentation.

\bibliographystyle{parts/siamplain}
\bibliography{}
\end{document}